\documentclass{amsart}
\usepackage{amsmath,amscd,amssymb,amsthm,amsbsy,amscd,stmaryrd}
\usepackage[dvips]{graphicx,color}
\usepackage{epsfig}
\usepackage{mathrsfs}
\usepackage{mathpazo}
\input{xy}

\newtheorem{theorem}{Theorem}[section]
\newtheorem{lemma}[theorem]{Lemma}
\newtheorem{proposition}[theorem]{Proposition}

\theoremstyle{definition}

\newcommand{\bA}{\boldsymbol{A}}
\newcommand{\bB}{\boldsymbol{B}}

\newcommand{\bT}{\boldsymbol{T}}

\numberwithin{equation}{section}

\begin{document}
\title{Local curvature estimates for the Laplacian flow}
\author{Yi Li}
\address{Faculty of Science, Technology and Communication (FSTC), Mathematic Research Unit\\
Campus Belval, Universite du Luxembourg\\
Maison du Nombre, 6, avenue de la Fonte\\
L-4364, Esch-sur-Alzette, Grand-Duchy of Luxembourg}
\email{yilicms@gmail.com}



\maketitle

\begin{abstract} In this paper we give local curvature estimates for the Laplacian
flow on closed $G_{2}$-structures under the condition that the Ricci
curvature is bounded along the flow. The main ingredient consists of
the idea of Kotschwar-Munteanu-Wang \cite{KMW2016} who gave local curvature estimates for the Ricci flow on complete manifolds and then provided a new elementary proof of Sesum's result \cite{Sesum2005}, and the particular structure of the Laplacian flow on closed $G_{2}$-structures. As an immediate consequence, this estimates give a new proof of Lotay-Wei's \cite{LW2017} result which is an analogue of Sesum's theorem.


The second result is about an interesting evolution equation for the
scalar curvature of the Laplacian flow of closed $G_{2}$-structures.
Roughly speaking, we can prove that the time derivative of the scalar
curvature $R_{t}$ is equal to the Laplacian of $R_{t}$, plus an extra
term which can be written as the difference of two nonnegative quantities.
\end{abstract}

\renewcommand{\labelenumi}{Case \theenumi.} %
\newtheoremstyle{mystyle}{3pt}{3pt}{\itshape}{}{\bfseries}{}{5mm}{} %
\theoremstyle{mystyle} \newtheorem{Thm}{Theorem} \theoremstyle{mystyle} %
\newtheorem{lem}{Lemma} \newtheoremstyle{citing}{3pt}{3pt}{\itshape}{}{%
\bfseries}{}{5mm}{\thmnote{#3}} \theoremstyle{citing} %
\newtheorem*{citedthm}{}


\section{Introduction}\label{section1}

Let $\mathcal{M}$ be a smooth $7$-manifold. The Laplacian flow for closed $G_{2}$-structures on $\mathcal{M}$ introduced by Bryant \cite{Bryant2005} is to study the torsion-free $G_{2}$-structures
\begin{equation}
\partial_{t}\varphi_{t}=\Delta_{\varphi_{t}}\varphi_{t}
, \ \ \ \varphi_{0}=\varphi,\label{1.1}
\end{equation}
where $\Delta_{\varphi_{t}}\varphi_{t}=dd^{\ast}_{\varphi_{t}}
\varphi_{t}+d^{\ast}_{\varphi_{t}}
d\varphi_{t}$ is the Hodge Laplacian of $g_{\varphi_{t}}$ and $\varphi$ is an initial closed
$G_{2}$-structure. Since $d\partial_{t}\varphi_{t}=\partial_{t}d\Delta_{\varphi_{t}}\varphi_{t}
=0$, we see that the flow (\ref{1.1}) preserves the closedness of
$\varphi_{t}$. For more background on $G_{2}$-structures, see Section \ref{section2}.
When $\mathcal{M}$ is compact, the flow (\ref{1.1}) can be viewed as the gradient flow for the
Hitchin functional introduced by Hitchin \cite{Hitchin2000}
\begin{equation}
\mathscr{H}: [\overline{\varphi}]_{+}\longrightarrow\mathbb{R}^{+}, \ \ \
\varphi\longmapsto\frac{1}{7}\int_{\mathcal{M}}\varphi\wedge\psi=\int_{\mathcal{M}}
\ast_{\varphi}1.\label{1.2}
\end{equation}
Here $\overline{\varphi}$ is a closed $G_{2}$-structure on
$\mathcal{M}$ and $[\overline{\varphi}]_{+}$ is the open subset of the cohomology class $[\overline{\varphi}]$ consisting of
$G_{2}$-structures. Any critical point of $\mathscr{H}$ gives a
torsion-free $G_{2}$-structure.

The study of Laplacian flows on some special $7$-manifolds, Laplacian solitons, and other flows on $G_{2}$-structures can be found in \cite{FFM2015, FY2017, Grigorian2013, Grigorian2016, HWY2017,KMT2012, Lin2013, LW2019, LW2020, WW2012a, WW2012b}.

Recently, Donaldson \cite{D1, D2, D3, D4} studied the co-associative Kovalev-Lefschetz
fibrations $G_{2}$-manifolds and $G_{2}$-manifolds with boundary.

\subsection{Notions and conventions}\label{subsection1.1}

To state the main results, we fix our notions used throughout this paper. Let $\mathcal{M}$ be as before a smooth $7$-manifold. The space of smooth functions
and the space of smooth vector fields are denoted respectively
by $C^{\infty}(\mathcal{M})$ and $\mathfrak{X}(\mathcal{M})$. The space of
$k$-tenors (i.e., $(0,k)$-covariant tensor fields) and $k$-forms on $\mathcal{M}$ are denoted, respectively, by $\otimes^{k}(\mathcal{M})
=C^{\infty}(\otimes^{k}(T^{\ast}\mathcal{M}))$ and
$\wedge^{k}(\mathcal{M})=C^{\infty}(\wedge^{k}(T^{\ast}
\mathcal{M}))$. For any $k$-tensor field $\boldsymbol{T}\in\otimes^{k}(\mathcal{M})$, we locally have the
expression $\bT=\bT_{i_{1}\cdots i_{k}}dx^{i_{1}}\otimes\cdots\otimes dx^{i_{k}}
=:\bT_{i_{1}\cdots i_{k}}dx^{i_{1}\otimes\cdots\otimes i_{k}}$. A $k$-form $\alpha$ on $\mathcal{M}$ can be written in the {\it
standard form} as $\alpha=\frac{1}{k!}\alpha_{i_{1}\cdots i_{k}}dx^{i_{1}}\wedge\cdots\wedge
dx^{i_{k}}=:\frac{1}{k!}\alpha_{i_{1}\cdots i_{k}}dx^{i_{1}
\wedge\cdots\wedge i_{k}}$, where $\alpha_{i_{1}\cdots i_{k}}$ is fully skew-symmetric in its
indices. Using the standard
forms, if we take the interior product $X\lrcorner\alpha$ of a $k$-form $\alpha\in\wedge^{k}(\mathcal{M})$ with a vector field $X\in\mathfrak{X}(\mathcal{M})$, we obtain the $(k-1)$-form $X\lrcorner\alpha=\frac{1}{(k-1)!}X^{m}\alpha_{mi_{1}\cdots i_{k-1}}
dx^{i_{1}\wedge\cdots\wedge i_{k-1}}$ which is also in the standard form. In particular, consider the vector space $\otimes^{2}(\mathcal{M})$ of $2$-tensors. For any $2$-tensor $\bA=\bA_{ij}dx^{i\otimes j}$, define $\bA^{\odot}:=\frac{1}{2}(\bA_{ij}+\bA_{ji})dx^{i\otimes j}\equiv\bA^{\odot}_{ij}
dx^{i\otimes j}$ and $\bA^{\wedge}:=\frac{1}{2}(\bA_{ij}-\bA_{ji})
dx^{i\otimes j}\equiv\bA^{\wedge}_{ij}dx^{i\otimes j}$. Then $\bA^{\odot}$ is an element of $\odot^{2}(\mathcal{M})$, the space of
symmetric $2$-tensors. Since\footnote{In our convention, for any $2$-form $\alpha=\frac{1}{2}\alpha_{ij}dx^{ij}$, we have
\begin{equation*}
\alpha\left(\frac{\partial}{\partial x^{k}},
\frac{\partial}{\partial x^{\ell}}\right)
=\frac{1}{2}\alpha_{ij}\left(dx^{i\otimes j}-dx^{j\otimes i}
\right)\left(\frac{\partial}{\partial x^{k}},
\frac{\partial}{\partial x^{\ell}}\right)
=\frac{1}{2}\alpha_{ij}\left(\delta^{i}_{k}\delta^{j}_{\ell}
-\delta^{j}_{k}\delta^{i}_{\ell}\right)=\frac{1}{2}\left(
\alpha_{k\ell}-\alpha_{\ell k}\right)=\alpha_{k\ell}
\end{equation*}
which justifies the notion $\alpha_{k\ell}$ as $\alpha(\partial/\partial x^{k},\partial/\partial x^{\ell})$. In general, for any $k$-form $\alpha=\frac{1}{k!}\alpha_{i_{1}\cdots i_{k}}dx^{i_{1}
\wedge\cdots\wedge i_{k}}$ we have $\alpha_{i_{1}\cdots i_{k}}=\alpha(\partial/\partial x^{i_{1}},\cdots,\partial/\partial x^{i_{k}})$, because $dx^{i_{1}\wedge\cdots\wedge i_{k}}=\sum_{\sigma\in\mathfrak{S}_{k}}{\rm sgn}(\sigma)dx^{i_{\sigma(1)}\otimes \cdots\otimes i_{\sigma(k)}}$.}$dx^{i\wedge j}=dx^{i\otimes j}-dx^{j\otimes i}$, it
follows that $\bA^{\wedge}=\frac{1}{2}\bA_{ij}dx^{i\wedge j}$. Define $\alpha^{\bA}:=\frac{1}{2}\alpha^{\bA}_{ij}dx^{i\wedge j}$ with $\alpha^{\bA}_{ij}:=
\bA_{ij}$. Then we see that $\alpha^{\bA}=\bA^{\wedge}\in\wedge^{2}(\mathcal{M})$ and $\otimes^{2}(\mathcal{M})=\odot^{2}(\mathcal{M})
\oplus\wedge^{2}(\mathcal{M})$.
\\

A given Riemannian metric $g$ on $\mathcal{M}$ determines two
isomorphisms between vector fields and $1$-forms: $\flat_{g}: \mathfrak{X}(\mathcal{M})\longrightarrow\wedge^{1}
(\mathcal{M})$ and $\sharp_{g}: \wedge^{1}(\mathcal{M})\longrightarrow\mathfrak{X}(\mathcal{M})$, where, for every vector field $X=X^{i}\frac{\partial}{\partial x^{i}}$ and $1$-form $\alpha=\alpha_{i}dx^{i}$, $\flat_{g}(X)=X^{i}g_{ij}dx^{j}\equiv
 X_{j}dx^{j}$ and $\sharp_{g}(\alpha)
=\alpha_{i}g^{ij}\frac{\partial}{\partial x^{j}}\equiv\alpha^{j}\frac{\partial}{\partial x^{j}}$. Using these two natural maps, we can frequently raise or lower indices on tensors. The metric $g$ also induces a metric on $k$-forms $g(dx^{i_{1}\wedge\cdots\wedge i_{k}},dx^{j_{1}
\wedge \cdots\wedge j_{k}})=
\det(g(dx^{i_{a}},dx^{j_{b}}))=\sum_{\sigma\in\mathfrak{S}_{7}}{\rm sgn}(\sigma)g^{i_{1}j_{\sigma(1)}}
\cdots g^{i_{k}j_{\sigma(k)}}$ where $\mathfrak{S}_{7}$ is the group of permutations of seven letters and ${\rm sgn}(\sigma)$ denotes the sign $(\pm1)$ of an element $\sigma$ of $\mathfrak{S}_{7}$. The inner product $\langle\cdot,
\cdot\rangle_{g}$ of two $k$-forms $\alpha,\beta\in\wedge^{k}
(\mathcal{M})$ now is given by $\langle\alpha,\beta\rangle_{g}=\frac{1}{k!}
\alpha_{i_{1}\cdots i_{k}}\beta^{i_{1}\cdots i_{k}}
=\frac{1}{k!}\alpha_{i_{1}\cdots i_{k}}\beta_{j_{1}\cdots j_{k}}
g^{i_{1}j_{1}}\cdots g^{i_{k}j_{k}}$.

Given two $2$-tensors $\bA, \bB\in
\otimes^{2}(\mathcal{M})$, with the forms $\bA=\bA_{ij}dx^{i\otimes j}$ and $\bB=\bB_{ij}dx^{i\otimes j}$. Define $\langle\langle\bA, \bB\rangle\rangle_{g}:=\bA_{ij}\bB^{ij}$. There are two special cases which will be used later:

\begin{itemize}

\item[(1)] $\alpha=\frac{1}{2}\alpha_{ij}dx^{i\wedge j}\in\wedge^{2}(\mathcal{M})$
and $\bB=\bB_{ij}dx^{i\otimes j}\in\otimes^{2}(\mathcal{M})$. In this case, $\alpha$ can be written as a $2$-tensor $\bA^{\alpha}=\bA^{\alpha}_{ij}dx^{i\otimes j}$ with $\bA^{\alpha}_{ij}
=\alpha_{ij}$. Then $\langle\langle\alpha,\bB\rangle\rangle_{g}:=\langle\langle\bA^{\alpha},
\bB\rangle\rangle_{g}=\alpha_{ij}\bB^{ij}$.

\item[(2)] $\alpha=\frac{1}{2}\alpha_{ij}dx^{i\wedge j}$ and $\beta=\frac{1}{2}\beta_{ij}
dx^{i\wedge j}\in\wedge^{2}(\mathcal{M})$. In this case, $\alpha, \beta$ can be both written as $2$-tensors $\bA^{\alpha}=\bA^{\alpha}_{ij}
dx^{i\otimes j}$ and $\bB^{\beta}=\bB^{\beta}_{ij}dx^{i\otimes j}$ with $\bA^{\alpha}_{ij}
=\alpha_{ij}$ and $\bB^{\beta}_{ij}=\beta_{ij}$. Then $\langle\langle\alpha,\beta\rangle\rangle_{g}:=\langle\langle\bA^{\alpha},
\bB^{\beta}\rangle\rangle_{g}
=\alpha_{ij}\beta^{ij}=2\langle\alpha,\beta\rangle_{g}$.

\end{itemize}

The norm of $\bA\in\otimes^{2}(\mathcal{M})$ is defined by $||\bA||^{2}_{g}:=\langle\langle\bA,\bA\rangle\rangle_{g}
=\bA_{ij}\bA^{ij}$, while the norm of $\alpha\in\wedge^{k}
(\mathcal{M})$ is $|\alpha|^{2}_{g}:=\langle\alpha,\alpha\rangle_{g}
=\frac{1}{k!}\alpha_{i_{1}\cdots i_{k}}\alpha^{i_{1}\cdots i_{k}}$. In particular, $||X||^{2}_{g}=X_{i}X^{i}=|\flat_{g}(X)|^{2}_{g}$ and $||\alpha||^{2}_{g}=2|\alpha|^{2}_{g}$, for any vector field $X\in\mathfrak{X}(\mathcal{M})$ and
$2$-form $\alpha$.
\\

The Levi-Civita connection associated to a given Riemannian metric $g$ is denoted by ${}^{g}\nabla$ or simply $\nabla$. Our convention on Riemann curvature tensor is $R^{m}_{ijk}\frac{\partial}{\partial x^{m}}$ $:={\rm Rm}(\frac{\partial}{\partial x^{i}},\frac{\partial}{\partial x^{j}}
)\frac{\partial}{\partial x^{k}}=(\nabla_{i}\nabla_{j}
-\nabla_{j}\nabla_{i})\frac{\partial}{\partial x^{k}}$ and $R_{ijk\ell}:=R^{m}_{ijk}g_{m\ell}$. The Ricci curvature of $g$ is given by $R_{jk}:=R_{ijk\ell}g^{i\ell}$. We use $dV_{g}$ and $\ast_{g}$ to denote the volume form and Hodge
star operator, respectively, on $\mathcal{M}$ associated to a metric $g$ and an orientation.
\\

We use the standard notion $A\ast B$ to denote some linear combination
of contractions of the tensor product $A\otimes B$ relative to the
metric $g_{t}$ associated the $\varphi_{t}$. In Theorem \ref{t1.4}
and its proof, all universal constants
$c, C$ below depend only on the given real number $p$.

\subsection{Main results}\label{subsection1.2}

Applying De Turck's trick and Hamilton's
Nash-Moser inverse function theorem, Bryant and Xu \cite{Bryant-Xu2011} proved the following
local time existence for (\ref{1.1}).

\begin{theorem}\label{t1.1}{\bf (Bryant-Xu \cite{Bryant-Xu2011})} For a compact $7$-manifold $\mathcal{M}$, the initial value problem (\ref{1.1}) has a unique solution for a short time interval $[0,T_{\max})$ with the maximal time $T_{\max}\in(0,\infty]$ depending on $\varphi$.
\end{theorem}

As in the Ricci flow, we can prove following results on the long
time existence for the Laplacian flow (\ref{1.1}).

\begin{theorem}\label{t1.2}{\bf (Lotay-Wei \cite{LW2017})} Let $\mathcal{M}$ be a compact $7$-manifold and $\varphi_{t}$, $t\in[0,T)$, where $T<\infty$, be a solution to the flow (\ref{1.1}) for closed $G_{2}$-structures with associated metric $g_{t}=g_{\varphi_{t}}$ for each $t$.

\begin{itemize}

\item[(a)] If the velocity of the flow satisfies
\begin{equation*}
\sup_{\mathcal{M}\times[0,T)}||\Delta_{t}\varphi_{t}
||_{t}<\infty,
\end{equation*}
then the solution $\varphi_{t}$ can be extended past time $T$.

\item[(b)] If $T=T_{\max}$, then
\begin{equation*}
\lim_{t\to T_{\max}}\sup_{\mathcal{M}}\left(||{\rm Rm}_{t}||^{2}_{t}
+||\nabla_{t}\bT_{t}||^{2}_{t}\right)=\infty.
\end{equation*}
Here $\bT_{t}$ is the torsion of $\varphi_{t}$ (see (\ref{2.14})).
\end{itemize}

\end{theorem}

In this paper, we give a new elementary proof of Theorem \ref{t1.2},
based on the idea of \cite{KMW2016} and the structure of the equation (\ref{1.1}).

\begin{theorem}\label{t1.3} Let $\mathcal{M}$ be a compact $7$-manifold and $
\varphi_{t}$, $t\in[0,T)$, where $T<\infty$, be a solution to
the flow (\ref{1.1}) for closed $G_{2}$-structures with associated
metric $g_{t}=g_{\varphi_{t}}$ for each $t$. Suppose that 
\begin{equation*}
K:=\sup_{\mathcal{M}\times[0,T)}
||{\rm Ric}_{t}||_{t}<\infty, \ \ \ \Lambda:=\sup_{\mathcal{M}}
||{\rm Rm}_{0}||_{0}.
\end{equation*}
Then
\begin{equation*}
\sup_{\mathcal{M}\times[0,T)}||{\rm Rm}_{t}||_{t}<\infty,
\end{equation*}
where the bound depends only on $n, K, T$ and $\Lambda$.

\end{theorem}

When $\mathcal{M}$ is compact, the theorem immediately implies
the part (a) in Theorem \ref{t1.2}. Indeed, we shall show that (see (\ref{3.18}) and (\ref{3.37}))
\begin{equation*}
\sup_{\mathcal{M}
\times[0,T)}||\Delta_{t}
\varphi_{t}||_{t}<\infty\Longleftrightarrow\sup_{\mathcal{M}
\times[0,T)}||{\rm Ric}_{t}||_{t}<\infty.
\end{equation*}
In the compact case, Theorem \ref{t1.3} shows that, if the conclusion in part (a) does not hold, then $T=T_{\max}$ and $\sup_{\mathcal{M}
\times[0,T_{\max})}||{\rm Rm}_{t}||_{t}<\infty$ which implies $\sup_{
\mathcal{M}\times[0,T_{\max})}$ $
(||{\rm Rm}_{t}||^{2}_{t}+||\nabla_{t}\bT_{t}||^{2}_{t})<\infty$, since the norm $||\nabla_{t}\bT_{t}||^{2}_{t}$ can be controlled by
$||{\rm Rm}_{t}||^{2}_{t}$ (see (\ref{3.63})). However, by part (b) in Theorem \ref{t1.2},
it is impossible. Therefore, the conclusion in part (a) is true.

As remarked in \cite{KMW2016}, to prove Theorem \ref{t1.3}, it suffices to
establish the following integral estimate.

\begin{theorem}\label{t1.4} Let $\mathcal{M}$ be a smooth $7$-manifold and $\varphi_{t}$, $t\in[0,T)$, where $T<\infty$, be a solution to the flow (\ref{1.1}) for closed $G_{2}$-structures with associated metric $g_{t}
=g_{\varphi_{t}}$ for each $t$. Assume that there exist constants $A, K>0$ and a point $x_{0}\in\mathcal{M}$ such that the geodesic ball $B_{g_{0}}(x_{0}, A/\sqrt{K})$ is compactly contained in $\mathcal{M}$ and that
\begin{equation*}
|{\rm Ric}_{t}|_{t}\leq K \ \ \ \text{on} \
B_{g_{0}}\left(x_{0},\frac{A}{\sqrt{K}}\right)
\times[0,T].
\end{equation*}
Then, for any $p\geq5$, there exists $c=c(p)>0$ so that
\begin{eqnarray}
\int_{B_{g_{0}}(x_{0},A/2\sqrt{K})}
||{\rm Rm}_{t}||^{p}_{t}dV_{t}&\leq&c(1+K)
e^{cKT}\int_{B_{g_{0}}(x_{0},A/\sqrt{K})}
||{\rm Rm}_{0}||^{p}_{0}dV_{0}\nonumber\\
&&+ \ c K^{p}\left(1+A^{-2p}\right) e^{cKT}
{\rm vol}_{t}\left(B_{g_{0}}
\left(x_{0},\frac{A}{\sqrt{K}}\right)
\right)\label{1.3}
\end{eqnarray}
for all $t\in[0,T]$.
\end{theorem}

Now by the standard De Giorgi-Nash-Moser iteration (our manifold is compact 
and the Ricci curvature is uniformly bounded),
under the condition in Theorem \ref{t1.4}, we can prove
\begin{equation}
||{\rm Rm}_{T}||_{T}(x_{0})
\leq d_{1}(d_{2}+\Lambda_{0}),\label{1.4}
\end{equation}
where $d_{1}, d_{2}$ are constants depending on $K, T, A$, and
\begin{equation*}
\Lambda_{0}:=\sup_{B_{g_{0}}(x_{0},A/\sqrt{K})}
||{\rm Rm}_{0}||_{0}.
\end{equation*}
Actually, this follows from the same argument in \cite{KMW2016} by noting that
\begin{equation}
(\Delta_{t}
-\partial_{t})||{\rm Rm}_{t}||_{t}\geq-c||{\rm Rm}_{t}||^{2}_{t}.
\label{1.5}
\end{equation}
To verify (\ref{1.5}), we use (\ref{2.26}), (\ref{3.61}) and
(\ref{3.65}) to
deduce that $||\nabla_{t} \bT_{t}||\leq c||{\rm Rm}_{t}||_{t}$ and
\begin{equation*}
||\nabla^{2}_{t}\bT_{t}||_{t}
\leq c||\nabla_{t}{\rm Rm}_{t}||_{t}+c||{\rm Rm}_{t}||^{3/2}_{t}.
\end{equation*}
Then, by (\ref{3.31}) and the Cauchy inequality
\begin{eqnarray*}
||\nabla_{t}{\rm Rm}_{t}||^{2}_{t}
&\leq&-\frac{1}{2}(\partial_{t}-\Delta_{t})||{\rm Rm}_{t}||^{2}_{t}
+c||{\rm Rm}_{t}||^{3}_{t}+c||{\rm Rm}_{t}||^{3/2}_{t}
||\nabla_{t}{\rm Rm}_{t}||_{t}\\
&\leq&-\frac{1}{2}(\partial_{t}-\Delta_{t})||{\rm Rm}_{t}||^{2}_{t}
+c||{\rm Rm}_{t}||^{3}_{t}+||\nabla_{t}{\rm Rm}_{t}||^{2}_{t}
\end{eqnarray*}
which implies (\ref{1.5}). Now the estimate (\ref{1.4}) yields Theorem
\ref{t1.3}.
\\

The analogue of Theorem \ref{t1.2} in the Ricci flow was proved
by Hamilton \cite{Hamilton82} (for part (b)) and Sesum \cite{Sesum2005} (for part (a)).
It is an open question (due to Hamilton, see \cite{Cao2011}) that the Ricci flow
will exist as long as the scalar curvature remains bounded. For the K\"ahler-Ricci
flow \cite{Zhang2010} or type-I Ricci flow \cite{EMT2011}, this question was
settled. For the general case, some partial result on Hamilton's conjecture was carried
out in \cite{Cao2011}.

For the Ricci-harmonic flow introduce by List \cite{List2005, List2008} (see also,
\cite{Muller2009, Muller2012}), the analogue of Theorem \ref{t1.2} was proved in \cite{List2005, List2008} (see also, \cite{Muller2009, Muller2012}) and \cite{CZ2013} (see \cite{LY2018} for another
proof). The author \cite{LY2, LY3} extended Cao's result \cite{Cao2011} to
the Ricci-harmonic flow. The same Hamilton's conjecture was asked by the author
in \cite{LY2, LY3}.

We can ask the same question for the Laplacian flow on closed $G_{2}$-structures.
In \cite{LW2017} (see Page 171, line -6 to -3, or Open Problem (3) in Page 230), Lotay and Wei asked that whether the Laplacian flow
on closed $G_{2}$-structures will exist as long as the torsion tensor or scalar
curvature remains bounded. Let $g_{t}$ be the associated metric of $\varphi_{t}$. Then the
evolution equation for $g_{t}$ is given by
\begin{equation}
\partial_{t}g_{ij}=-2R_{ij}-\frac{4}{3}|\bT_{t}|^{2}_{t}
g_{ij}
-4\bT_{i}{}^{k}\bT_{kj}.\label{1.6}
\end{equation}
For the Laplacian flow on closed $G_{2}$-structures, the torsion $\bT_{t}$ is actually a $2$-form for each $t$, hence we use the norm $|\cdot|_{t}$
in (\ref{1.6}). The standard formula for the scalar
curvature $R_{t}$ gives (see (\ref{3.23}))
\begin{equation}
\partial_{t}R_{t}=\Delta_{t}R_{t}
+2||{\rm Ric}_{t}||^{2}_{t}-\frac{2}{3}R^{2}_{t}
+4R_{ijk\ell}\bT^{ik}\bT^{j\ell}+4(\nabla^{j}\bT^{ik})
(\nabla_{i}\bT_{jk}).\label{1.7}
\end{equation}
Now the above mentioned open problem states that
\begin{equation*}
\text{Is it ture that} \ \lim_{t\to T_{\max}}R_{t}=-\infty?
\end{equation*}
The ``minus infinity'' comes from the fact that along the Laplacian flow on
closed $G_{2}$-structures the scalar curvature is always nonpositive
(see (\ref{2.26})). The following Proposition \ref{p1.5} is motivate to
solve this problem, and starts from the basic evolution
equation (\ref{1.7}) where the last two terms on the right-hand side
do not have good signature. However, using the closedness of
$\varphi_{t}$ (in particular, the identity (\ref{3.23})), we can
prove the following interesting evolution equation for $R_{t}$.

\begin{proposition}\label{p1.5} Let $\mathcal{M}$ be a smooth $7$-manifold and $
\varphi_{t}$, $t\in[0,T)$, where $T\in(0,\infty]$, be a solution to the flow (\ref{1.1}) for closed $G_{2}$-structures with associated metric $g_{t}=
g_{\varphi_{t}}$ for each $t$. Then the scalar curvature $R_{t}$ satisfies
\begin{eqnarray}
\partial_{t}R_{t}&=&\Delta_{t}R_{t}+\bigg\{2\left|\left|R_{ij}+\frac{2}{3}
|\bT_{t}|^{2}_{t}g_{ij}\right|\right|^{2}_{t}
+\frac{1}{2}\left|\left|R_{ijab}R^{ij}{}_{mn}
-\psi_{abmn}\right|\right|^{2}_{t}\nonumber\\
&&+ \ \frac{1}{2}
\left|\left|2\bT_{ia}\bT_{jb}R^{ij}{}_{mn}
-\psi_{abmn}\right|\right|^{2}_{t}
+\frac{1}{2}\left|\left|2\widehat{\bT}_{am}\widehat{\bT}_{bn}
-\psi_{abmn}\right|\right|^{2}_{t}\nonumber\\
&&+ \ 2||\widehat{\bT}_{t}||^{2}_{t}+4||\nabla_{t}\bT_{t}||^{2}_{t}\bigg\}-\bigg\{||{\rm Rm}_{t}||^{2}_{t}
+\frac{26}{9}R^{2}_{t}+\frac{1}{2}
\left|\left|R_{ijab}R^{ij}{}_{mn}\right|\right|^{2}_{t}\label{1.8}\\
&&+ \ 2\left|\left|\bT_{ia}\bT_{jb}R^{ij}{}_{mn}
\right|\right|^{2}_{t}
+2||\widehat{\bT}_{t}||^{4}_{t}+210\bigg\}.\nonumber
\end{eqnarray}
Here $\widehat{\bT}_{ij}=\bT_{i}{}^{k}\bT_{kj}$.
\end{proposition}

Observe that the above well-arranged evolution equation can give us a weakly lower bound
for $R_{t}$, which can not prove or disprove the conjecture of Lotay and Wei.
\\

We give an outline of the current paper. We review the basic theory in
Section \ref{section2} about $G_{2}$-structures, $G_{2}$-decompositions of $2$-forms and $3$-forms, and general flows on $G_{2}$-structures.
In Section \ref{section3}, we rewrite results in Section \ref{section2} for closed $G_{2}$-structures, and the local curvature estimates will be given
in the last subsection.

\subsection{Acknowledgments}\label{subsection1.3}

The author is supported in part by the Fonds National
de la Recherche Luxembourg (FNR) under the OPEN scheme (project GEOMREV O14/7628746).
\\

The main result was carried out during the Young Geometric
Analysts Forum 2018, 29th January -- 2th February, in Tsinghua Sanya International Mathematics
Forum.

The author, together with other six friends, thanks Yunhui Wu who
personally provided us 14, the dimension of $G_{2}$, very fresh coconuts during the forum.
\\

The author thanks Joel Fine, Brett Kotschwar, Chengjian Yao, Yong Wei, 
and Anton Thalmaier for useful discussion on
the Laplacian flows and the earlier version of this paper. He also thanks Jason Lotay for his interested in this paper.

\section{Basic theory of $G_{2}$-structures}\label{section2}

In this section, we view some basic theory of $G_{2}$-structures, following \cite{Bryant2005, Karigiannis2003, Karigiannis2005, Karigiannis2006, Karigiannis2007, LW2017}. Let $\{e_{1},\cdots, e_{7}\}$ denote the standard basis of $\mathbb{R}^{7}$ and let $\{e^{1},\cdots, e^{7}\}$ be its dual
basis. Define the $3$-form
\begin{equation*}
\phi:=e^{1\wedge2\wedge3}+e^{1\wedge4\wedge5}+e^{1\wedge6\wedge7}
+e^{2\wedge4\wedge6}-e^{2\wedge5\wedge7}-e^{3\wedge4\wedge7}
-e^{3\wedge5\wedge6},
\end{equation*}
where $e^{
i\wedge j\wedge k}:=
e^{i}\wedge e^{j}\wedge e^{k}$. The subgroup $G_{2}$, which fixes $\phi$, of ${\bf GL}(7,\mathbb{R})$ is the
$14$-dimensional Lie subgroup of ${\bf SO}(7)$, acts irreducibly on $\mathbb{R}^{7}$, and preserves the metric and orientation for which $\{e_{1},\cdots, e_{7}\}$ is an oriented orthonormal basis. Note that
$G_{2}$ also preserves the $4$-form
\begin{equation*}
\ast_{\phi}\phi=e^{4\wedge5\wedge6\wedge7}+e^{2\wedge3\wedge6
\wedge7}+e^{2\wedge3\wedge4\wedge5}+e^{1\wedge3\wedge5\wedge7}-e^{1\wedge3\wedge4\wedge6}-e^{1\wedge2\wedge5\wedge6}-e^{1\wedge2\wedge4\wedge7}.\label{23.2.2}
\end{equation*}
where the Hodge star operator $\ast_{\phi}$ is determined by the metric and orientation.
\\

For a smooth $7$-manifold $\mathcal{M}$ and a point $x\in\mathcal{M}$,
define as in \cite{LW2017}
\begin{equation*}
\wedge^{3}_{+}(T^{\ast}_{x}\mathcal{M}):=
\left\{\varphi_{x}\in\wedge^{3}(T^{\ast}_{x}\mathcal{M}): \begin{array}{cc}
\textsf{u}^{\ast}
\phi=\varphi_{x} \ \text{for some invertible}\\
\text{map} \ \textsf{u}\in{\rm Hom}_{\mathbb{R}}(T_{x}\mathcal{M}, \mathbb{R}^{7})
\end{array}
\right\}
\end{equation*}
and the bundle
\begin{equation*}
\wedge^{3}_{+}(T^{\ast}\mathcal{M}):=\bigsqcup_{x\in\mathcal{M}}
\wedge^{3}_{+}(T^{\ast}_{x}\mathcal{M}).
\end{equation*}
We call a section $\varphi$ of $\wedge^{3}_{+}(T^{\ast}\mathcal{M})$ a
{\it positive $3$-form} on $\mathcal{M}$ or a {\it $G_{2}$-structure} on $\mathcal{M}$, and denote the space of positive $3$-forms by $\wedge^{3}_{+}(\mathcal{M})$. The existence of $G_{2}$-structures is equivalent to the property that $\mathcal{M}$ is oriented and spin, which is equivalent to the vanishing
of the first and second Stiefel-Whitney classes. From the definition of $G_{2}$-structures, we see that any $\varphi\in\wedge^{3}_{+}(\mathcal{M})$ uniquely determines a Riemannian
metric $g_{\varphi}$ and an orientation $dV_{\varphi}$, hence the Hodge star operator $\ast_{\varphi}$ and the associated $4$-form
\begin{equation}
\psi:=\ast_{\varphi}\varphi.\label{2.1}
\end{equation}
We also have the isomorphisms $\flat_{\varphi}:=\flat_{g_{\varphi}}$ and $\sharp_{\varphi}:=
\sharp_{g_{\varphi}}$. For a given $G_{2}$-structure $\varphi\in\wedge^{3}_{+}(\mathcal{M})$, we denote by $\langle\cdot,\cdot\rangle_{\varphi}$, $\langle\langle\cdot,\cdot\rangle
\rangle$, $|\cdot|_{\varphi}$, $||\cdot||_{\varphi}$, the corresponding
inner products $\langle\cdot,\cdot\rangle_{g_{\varphi}}$, $\langle\langle
\cdot,\cdot\rangle\rangle_{g_{\varphi}}$ and norms $|\cdot|_{g_{\varphi}}$,
$||\cdot||_{g_{\varphi}}$.
\\

Given a $G_{2}$-structure $\varphi\in\wedge^{3}_{+}(\mathcal{M})$.
We say that $\varphi$ is {\it torsion-free} if $\varphi$ is parallel with
respect to the metric $g_{\varphi}$. Equivalently, $\varphi$ is torsion-free if and only if ${}^{\varphi}\nabla\varphi=0$, where ${}^{\varphi}
\nabla$ is the Levi-Civita connection of $g_{\varphi}$.

\begin{theorem}\label{t2.1}{\bf (Fern\'andez-Gray
\cite{Fernandez-Gray1982})} The $G_{2}$-structure $\varphi$ is torsion-free if and only if $\varphi$ is both closed (i.e., $d\varphi=0$) and co-closed (i.e., $d\ast_{\varphi}
\varphi=d\psi=0$).
\end{theorem}

When $\mathcal{M}$ is compact, the above theorem says that a $G_{2}$-structure $\varphi$ is torsion-free if and only if $\varphi$ is harmonic with respect
to the induces metric $g_{\varphi}$.

We say that a $G_{2}$-structure $\varphi$ is {\it closed} (resp., {\it co-closed}) if $d\varphi=0$ (resp., $d\psi=0$). Theorem \ref{t2.1} can be restated as that a $G_{2}$-structure is torsion-free if and only if it is both closed
and co-closed.

\subsection{$G_{2}$-decompositions of $\wedge^{2}(\mathcal{M})$ and $\wedge^{3}(\mathcal{M})$}
\label{section2.1}

A $G_{2}$-structure $\varphi$ induces splittings of the bundles $\wedge^{k}(T^{\ast}\mathcal{M})$, $2\leq k\leq 5$, into direct summands, which we denote by $\wedge^{k}_{\ell}
(T^{\ast}\mathcal{M},\varphi)$ with $\ell$ being the rank of the bundle. We let the space of sections of $\wedge^{k}_{\ell}(T^{\ast}\mathcal{M},\varphi)$ by $\wedge^{k}_{\ell}(\mathcal{M},\varphi)$. Define the natural projections
\begin{equation}
\pi^{k}_{\ell}: \wedge^{k}(\mathcal{M})\longrightarrow
\wedge^{k}_{\ell}(\mathcal{M},\varphi), \ \ \ \alpha\longmapsto
\pi^{k}_{\ell}(\alpha).\label{2.2}
\end{equation}
We mainly focus on the $G_{2}$--decompositions of $\wedge^{2}
(\mathcal{M})$ and $\wedge^{3}
(\mathcal{M})$. Recall that
\begin{eqnarray}
\wedge^{2}(\mathcal{M})&=&\wedge^{2}_{7}(\mathcal{M},\varphi)
\oplus\wedge^{2}_{14}(\mathcal{M},\varphi),\label{2.3}\\
\wedge^{3}(\mathcal{M})&=&\wedge^{3}_{1}(\mathcal{M},\varphi)
\oplus\wedge^{3}_{7}(\mathcal{M},\varphi)
\oplus\wedge^{3}_{27}(\mathcal{M},\varphi).\label{2.4}
\end{eqnarray}
Here each component is determined by
\begin{eqnarray*}
\wedge^{2}_{7}(\mathcal{M},\varphi)
&=&\{X\lrcorner\varphi: X\in\mathfrak{X}(\mathcal{M})\} \ \ = \ \ \{\beta\in\wedge^{2}(\mathcal{M}): \ast_{\varphi}(\varphi\wedge
\beta)=2\beta\},\\
\wedge^{2}_{14}(\mathcal{M},\varphi)
&=&\{\beta\in\wedge^{2}(\mathcal{M}): \psi\wedge\beta=0\} \ \ = \ \ \{\beta\in\wedge^{2}(\mathcal{M}): \ast_{\varphi}(\varphi\wedge\beta)=-\beta\},\\
\wedge^{3}_{1}(\mathcal{M},\varphi)&=&\{f\varphi: f\in C^{\infty}(\mathcal{M})\},\\
\wedge^{3}_{7}(\mathcal{M},\varphi)&=&\left\{\ast_{\varphi}(\varphi\wedge\alpha): \alpha\in\wedge^{1}(\mathcal{M})\right\} \ \ = \ \
\left\{X\lrcorner\psi: X\in\mathfrak{X}(\mathcal{M})
\right\},\\
\wedge^{3}_{27}(\mathcal{M},\varphi)&=&\{\eta\in\wedge^{3}(\mathcal{M}):
\eta\wedge\varphi=\eta\wedge\psi=0\}.
\end{eqnarray*}

For any $2$-form $\beta=\frac{1}{2}\beta_{ij}dx^{i\wedge j}
\in\wedge^{2}(\mathcal{M})$, its two components $\pi^{2}_{7}(\beta)$
and $\pi^{2}_{14}(\beta)$ are determined by
\begin{eqnarray}
\pi^{2}_{7}(\beta)&=&\frac{\beta+\ast_{\varphi}(\varphi\wedge\beta)}{3} \ \ = \ \ \frac{1}{2}\left(\frac{1}{3}\beta_{ab}+\frac{1}{6}\beta^{\ell m}
\psi_{\ell m ab}\right)dx^{ab},\label{2.5}\\
\pi^{2}_{14}(\beta)
&=&\frac{2\beta-\ast_{\varphi}(\varphi\wedge\beta)}{3} \ \ = \ \ \frac{1}{2}\left(\frac{2}{3}\beta_{ab}-\frac{1}{6}\beta^{\ell m}
\psi_{\ell mab}\right)dx^{ab}.\label{2.6}
\end{eqnarray}

To decompose $3$-forms, recall two maps introduce by Bryant
\cite{Bryant2005}
\begin{equation}
\textsf{i}_{\varphi}: \odot^{2}(\mathcal{M})\longrightarrow\wedge^{3}(\mathcal{M}), \ \ \
\textsf{j}_{\varphi}: \wedge^{3}(\mathcal{M})\longrightarrow\odot^{2}(\mathcal{M}),\label{2.7}
\end{equation}
where
\begin{eqnarray}
\textsf{i}_{\varphi}(h)&:=&h_{ij}g^{j\ell}dx^{i}
\wedge\left(\frac{\partial}{\partial x^{\ell}}\lrcorner\varphi
\right) \ \ = \ \ \frac{1}{2}h_{i\ell}\varphi^{\ell}{}_{jk}dx^{ijk}\nonumber\\
&=&\frac{1}{6}\left(h_{i\ell}\varphi^{\ell}{}_{jk}
 +h_{j\ell}\varphi_{i}{}^{\ell}{}_{k}+h_{k\ell}\varphi_{ij}{}^{\ell}
 \right)dx^{ijk}, \ \ \ h=h_{ij}dx^{ij}\in\odot^{2}(\mathcal{M}),
 \label{2.8}
\end{eqnarray}
and
\begin{equation}
\left(\textsf{j}_{\varphi}(\eta)\right)(X,Y):=
\ast_{\varphi}\left((X\lrcorner\varphi)\wedge(Y\lrcorner\varphi)
\wedge\eta\right).\label{2.9}
\end{equation}
Then $\textsf{i}_{\varphi}$ is injective and is isomorphic onto $\wedge^{3}_{1}
(\mathcal{M},\varphi)\oplus\wedge^{3}_{27}(\mathcal{M},\varphi)$, and $\textsf{j}_{\varphi}$ is an isomorphism between $\wedge^{3}_{1}(\mathcal{M},\varphi)\oplus\wedge^{3}_{27}(\mathcal{M},
\varphi)$ and $\odot^{2}(\mathcal{M})$. Moreover, for any $3$-form $\eta\in\wedge^{3}(\mathcal{M})$, we have
\begin{equation}
\eta=\textsf{i}_{\varphi}(h)+X\lrcorner\psi\label{2.10}
\end{equation}
for some symmetric $2$-tensor $h\in\odot^{2}(\mathcal{M})$ and vector field $X\in\mathfrak{X}(\mathcal{M})$. Then
\begin{eqnarray*}
\eta&=&h_{i}{}^{\ell}dx^{i}\wedge\left(\frac{\partial}{\partial x^{\ell}}\lrcorner
\varphi\right)+X^{\ell}\left(\frac{\partial}{\partial x^{\ell}}\lrcorner\psi\right) \ \ = \ \ \frac{1}{2}h_{i}{}^{\ell}\varphi_{\ell jk}dx^{ijk}+\frac{1}{6}X^{\ell}
\psi_{\ell ijk}dx^{ijk}\\
&=&\frac{1}{6}\left(3 h_{i}{}^{\ell}\varphi_{\ell jk}
+X^{\ell}\psi_{\ell ijk}\right)dx^{ijk} \ \ = \ \ \frac{1}{6}\eta_{ijk}
dx^{ijk}.
\end{eqnarray*}
Write $h$ as $h_{ij}=\mathring{h}_{ij}+\frac{1}{7}{\rm tr}_{\varphi}(h)\!\ g_{\varphi}$, where $\mathring{h}\in\odot^{2}_{0}(\mathcal{M})$ is the trace-free part
of $h$, one has
\begin{equation}
\eta=\underbrace{\frac{3}{7}\left({\rm tr}_{\varphi}(h)\right)\varphi}_{
\pi^{3}_{1}(\eta)}
+\underbrace{\frac{1}{2}\mathring{h}_{i}{}^{\ell}\varphi_{\ell jk}
dx^{ijk}}_{\pi^{3}_{27}(\eta)}+
\underbrace{\frac{1}{6}X^{\ell}\psi_{\ell ijk}dx^{ijk}}_{\pi^{3}_{7}(\eta)}.
\label{2.11}
\end{equation}

\subsection{The torsion tensors of a $G_{2}$-structure}
\label{subsection2.2}

By Hodge duality we obtain the $G_{2}$-decompositions of $4$-forms $\wedge^{4}(\mathcal{M})=\wedge^{4}_{1}(\mathcal{M},\varphi)
\oplus\wedge^{4}_{7}(\mathcal{M},\varphi)\oplus\wedge^{4}_{27}(\mathcal{M},
\varphi)$ and $5$-forms $
\wedge^{5}(\mathcal{M})=\wedge^{5}_{7}(\mathcal{M},\varphi)
\oplus\wedge^{5}_{14}(\mathcal{M},\varphi)$, respectively. By
definition, we can find forms $\tau_{0}\in C^{\infty}(\mathcal{M})$, $\tau_{1},
\widetilde{\tau}_{1}\in\wedge^{1}(\mathcal{M})$, $\tau_{2}\in\wedge^{2}_{14}(\mathcal{M},\varphi)$, and $\tau_{3}\in\wedge^{3}_{27}(\mathcal{M},\varphi)$ such that
\begin{equation}
d\varphi=\tau_{0}\psi+3\tau_{1}\wedge\varphi
+\ast_{\varphi}\tau_{3}, \ \ \ d\psi=4\widetilde{\tau}_{1}\wedge\psi
-\ast_{\varphi}\tau_{2}.\label{2.12}
\end{equation}
Since $\tau_{2}\in\wedge^{2}_{14}(\mathcal{M},\varphi)$, it follows that $\tau_{2}\wedge\varphi=-\ast_{\varphi}\tau_{2}$. Then (\ref{2.12}) can be written as in the sense of Bryant \cite{Bryant2005}
\begin{equation}
d\varphi=\tau_{0}\psi+3\tau_{1}\wedge\varphi
+\ast_{\varphi}\tau_{3}, \ \ \ d\psi=4\widetilde{\tau}_{1}\wedge\psi
+\tau_{2}\wedge\varphi.\label{2.13}
\end{equation}
It can be proved that $\tau_{1}=\widetilde{\tau}_{1}$ (see \cite{Karigiannis2007}). We call
$\tau_{0}$ the {\it scalar torsion}, $\tau_{1}$ the {\it vector torsion}, $\tau_{2}$ the {\it Lie algebra torsion}, and $\tau_{3}$ the {\it symmetric
traceless torsion}. We also call $\boldsymbol{\tau}_{\varphi}:=\{\tau_{0},\tau_{1},\tau_{2},
\tau_{3}\}$ the {\it intrinsic torsion forms} of the $G_{2}$-structure $\varphi$.

Recall that a $G_{2}$-structure $\varphi$ is torsion-free if and only if $d
\varphi=d\psi=0$ by Theorem \ref{t2.1}. From (\ref{2.12}) we see that $\varphi$ is torsion-free if and only if the intrinsic torsion forms $\boldsymbol{\tau}_{\varphi}\equiv=0$; that is, $\tau_{0}=\tau_{1}=\tau_{2}
=\tau_{3}=0$.

\begin{lemma}\label{l2.2} {\bf (Fern\'andez-Gray, \cite{Fernandez-Gray1982})} For any $X\in\mathfrak{X}(\mathcal{M})$, the $3$-form $\nabla_{X}
\varphi$ lines in the space $\wedge^{3}_{7}(\mathcal{M},\varphi)$. Therefore the covariant derivative $\nabla\varphi\in\wedge^{1}(\mathcal{M})
\otimes\wedge^{3}_{7}(\mathcal{M})$.
\end{lemma}

Consequently, there exists a $2$-tensor $\bT=\bT_{ij}dx^{i\otimes j}$, called the {\it full torsion tensor}, such that
\begin{equation}
\nabla_{\ell}\varphi=\bT_{\ell}{}^{n}\psi_{nabc}.\label{2.14}
\end{equation}
Equivalently,
\begin{equation}
\bT_{\ell m}=\frac{1}{24}(\nabla_{\ell}\varphi_{abc})
\psi_{m}{}^{abc}.\label{2.15}
\end{equation}
Write
\begin{eqnarray}
\tau_{1}&=&(\tau_{1})_{i}dx^{i}\in\wedge^{1}(\mathcal{M}),
\label{2.16}\\
\tau_{2}&=&\frac{1}{2}(\tau_{2})_{ab}dx^{ab}\in\wedge^{2}_{14}(\mathcal{M}),
\label{2.17}\\
\tau_{3}&=&\frac{1}{2}(\tau_{3})_{i}{}^{\ell}\varphi_{\ell ij}
dx^{ijk}\in\wedge^{3}_{27}(\mathcal{M},\varphi).\label{2.18}
\end{eqnarray}
The associated $2$-tensor $\boldsymbol{\tau}_{3}:=(\tau_{3})_{ij}dx^{i\otimes j}$ of $\tau_{3}$ lies in the space $\odot^{2}_{0}(\mathcal{M})$. With this convenience, the full torsion tensor $\bT_{\ell m}$ is determined by
\begin{equation}
\bT_{\ell m}=\frac{\tau_{0}}{4}g_{\ell m}-(\boldsymbol{\tau}_{3})_{\ell m}
-\left(\sharp_{\varphi}(\tau_{1})\lrcorner\varphi\right)_{\ell m}
-\frac{1}{2}(\tau_{2})_{\ell m}\label{2.19}
\end{equation}
or as $2$-tensors,
\begin{equation}
\bT=\frac{\tau_{0}}{4}g_{\varphi}
-\boldsymbol{\tau}_{3}-\sharp_{\varphi}(\tau_{1})\lrcorner\varphi
-\frac{1}{2}\tau_{2}.\label{2.20}
\end{equation}
Here the $2$-form $\sharp_{\varphi}(\tau_{1})\lrcorner\varphi$ is defined by \begin{equation*}
\sharp_{\varphi}(\tau_{1})\lrcorner\varphi
=\frac{1}{2}\left(\sharp_{\varphi}(\tau_{1})\lrcorner
\varphi\right)dx^{a\wedge b}=\frac{1}{2}\left((\tau_{1})_{k}\varphi^{k}{}_{ab}
\right)dx^{a\wedge b}.
\end{equation*}
As an application, this gives another proof of Theorem \ref{t2.1}.
\\

For fixed indices $i$ and $j$, set
\begin{equation}
R_{ij|k\ell}:=R_{ijk\ell} \ \text{is skew-symmetric in} \ k \
\text{and} \ \ell,\label{2.21}
\end{equation}
where
\begin{equation}
R_{ij|\bullet\bullet}:=\frac{1}{2}R_{ij|k\ell}dx^{k\ell}
=\frac{1}{2}R_{ijk\ell}dx^{k\ell}\in\wedge^{2}
(\mathcal{M}).\label{2.22}
\end{equation}
Then, according to (\ref{2.5}) and (\ref{2.6})
\begin{equation*}
R_{ijk\ell}=R_{ij|k\ell}=\left(\pi^{2}_{7}(R_{ij|\bullet\bullet})
\right)_{k\ell}+\left(\pi^{2}_{14}(R_{ij|\bullet\bullet})\right)_{k\ell},
\end{equation*}
where
\begin{eqnarray*}
\left(\pi^{2}_{7}(R_{ij|\bullet\bullet})\right)_{k\ell}
&=&\frac{1}{3}R_{ij|k\ell}+\frac{1}{6}R_{ij|ab}\psi^{ab}{}_{k\ell} \ \ = \ \
\frac{1}{3}R_{ijk\ell}+\frac{1}{6}R_{ijab}\psi^{ab}{}_{k\ell},\\
\left(\pi^{2}_{14}(R_{ij|\bullet\bullet})\right)_{k\ell}&=&
\frac{2}{3}R_{ij|k\ell}-\frac{1}{6}R_{ij|ab}\psi^{ab}{}_{k\ell} \ \ = \ \
\frac{1}{3}R_{ijk\ell}-\frac{1}{6}R_{ijab}\psi^{ab}{}_{k\ell}.
\end{eqnarray*}
Karigiannis \cite{Karigiannis2007} (see also the equivalent
formula obtained by Bryant in \cite{Bryant2005})
proved that the Ricci curvature is given by
\begin{eqnarray}
R_{jk}&=&R_{ijk\ell}g^{i\ell} \ \ = \ \ 3\left(\pi^{2}_{7}(R_{ij|\bullet\bullet})\right)_{k\ell}g^{i\ell} \ \ = \ \ \frac{3}{2}\left(\pi^{2}_{14}(R_{ij|\bullet\bullet})\right)_{k\ell}
g^{i\ell}\nonumber\\
&=&-\left(\nabla_{i}\bT_{jm}-\nabla_{j}\bT_{im}\right)\varphi^{m}{}_{
k}{}^{i}-\bT_{j}{}^{i}\bT_{ik}+\left({\rm tr}_{\varphi}\bT\right)\bT_{jk}
+\bT_{jb}\bT_{ia}\psi^{iab}{}_{k},\label{2.23}\\
&=&-\nabla_{i}\left(\bT_{j}{}^{n}\varphi_{nk}{}^{i}\right)
+\nabla_{j}\left(\bT_{i}{}^{n}\varphi_{nk}{}^{i}\right)
-\bT_{j}{}^{i}\bT_{ik}+\left({\rm tr}_{\varphi}\bT\right)
\bT_{jk}-\bT_{jb}\bT_{ia}\psi^{iab}{}_{k}.\nonumber
\end{eqnarray}
Cleyton and Ivanov \cite{Cleyton-Ivanov2007} also derived a formula for the Ricci tensor for closed $G_{2}$-structures in terms of
$d^{\ast}_{\varphi}\varphi$. Taking the trace of (\ref{2.23}), we
obtain Btyant's formula \cite{Bryant2005} for the scalar curvature
\begin{eqnarray}
\ \ R&=&-12\nabla^{\ell}(\tau_{1})_{\ell}
+\frac{21}{8}\tau^{2}_{0}-||\boldsymbol{\tau}_{3}||^{2}_{\varphi}
+5||\sharp_{\varphi}(\tau_{1})\lrcorner\varphi||^{2}_{\varphi}-\frac{1}{4}||\tau_{2}||^{2}_{\varphi},
\nonumber\\
&=&-12\nabla^{\ell}(\tau_{1})_{\ell}
+\frac{21}{8}\tau^{2}_{0}-||\boldsymbol{\tau}_{3}||^{2}_{\varphi}
+30|\tau_{1}|^{2}_{\varphi}-\frac{1}{2}|\tau_{2}|^{2}_{\varphi},
\label{2.24}
\end{eqnarray}

For a closed $G_{2}$-structure, we have $\tau_{0}=\tau_{1}=\tau_{3}=0$
and then $R=-\frac{1}{4}||\tau_{2}||^{2}_{\varphi}\leq0$. On the other hand, we have $(\tau_{2})_{ij}=-2\bT_{ij}$ by (\ref{2.20}). Thus the full torsion tensor $\bT$ is actually a $2$-form
\begin{equation}
\bT=\frac{1}{2}\bT_{ij}dx^{ij}\in\wedge^{2}(\mathcal{M})\label{2.25}
\end{equation}
and the scalar curvature can be written in terms of $T$
\begin{equation}
R=-||\bT||^{2}_{\varphi}=-2|\bT|^{2}_{\varphi}\leq0.\label{2.26}
\end{equation}
Hence, for closed $G_{2}$-structures, scalar curvatures are always non-positive.
\\

Finally, we mention a Bianchi type identity
\begin{equation}
\nabla_{i}\bT_{j\ell}-\nabla_{j}\bT_{i\ell}
=-\frac{1}{2}R_{ijab}\varphi^{ab}{}_{\ell}
-\bT_{ia}\bT_{jb}\varphi^{ab}{}_{\ell}
=-\left(\frac{1}{2}R_{ijab}+\bT_{ia}\bT_{jb}\right)
\varphi^{ab}{}_{\ell}.\label{2.27}
\end{equation}
The proof can be found in \cite{Karigiannis2007}.

\subsection{General flows on $G_{2}$-structures}\label{section2.3}

For any family $(\varphi_{t})_{t}$ of $G_{2}$-structures, according to the decomposition
(\ref{2.10}), we can consider the general flow
\begin{equation}
\partial_{t}\varphi_{t}=\textsf{i}_{\varphi_{t}}(h_{t})
+X_{t}\lrcorner\psi_{t}\label{2.28}
\end{equation}
where $h_{t}\in\odot^{2}(\mathcal{M})$ and $X_{t}\in\mathfrak{X}
(\mathcal{M})$. The general flow (\ref{2.28}) locally can be written as
\begin{equation}
\partial_{t}\varphi_{ijk}=h_{i}{}^{\ell}\varphi_{\ell jk}
+h_{j}{}^{\ell}\varphi_{i\ell k}+h_{k}{}^{\ell}\varphi_{ij\ell}
+X^{\ell}\psi_{\ell ijk}.\label{2.29}
\end{equation}
We write for $g_{t}$ and $dV_{t}$ the metric and
volume form associated to $\varphi_{t}$, respectively.

\begin{theorem}\label{t2.3} Under the general flow (\ref{2.28}), we have
\begin{eqnarray}
\partial_{t}g_{ij}&=&2h_{ij}\label{2.30},\\
\partial_{t}g^{ij}&=&-2h^{ij},\label{2.31}\\
\partial_{t}dV_{t}&=&\left({\rm tr}_{t}h_{t}\right)dV_{t},\label{2.32}\\
\partial_{t}\bT_{pq}&=&\bT_{p}{}^{m}h_{mq}
-\bT_{p}{}^{m}X^{k}\varphi_{kmq}-(\nabla_{k}h_{ip})\varphi^{ki}{}_{q}
+\nabla_{p}X_{q}.\label{2.33}
\end{eqnarray}
\end{theorem}

These evolution equations can be found in \cite{Karigiannis2007}.

\section{Laplacian flows on closed $G_{2}$-structures}\label{section3}

We now consider the Laplacian flow for closed $G_{2}$-structures
\begin{equation}
\partial_{t}\varphi_{t}=\Delta_{\varphi_{t}}\varphi_{t}=\Delta_{t}\varphi_{t}
, \ \ \ \varphi_{0}=\varphi,\label{3.1}
\end{equation}
where $\Delta_{\varphi_{t}}\varphi_{t}=dd^{\ast}_{\varphi_{t}}
\varphi_{t}+d^{\ast}_{\varphi_{t}}
d\varphi_{t}$ is the Hodge Laplacian of $g_{\varphi_{t}}$ and $\varphi$ is an initial closed
$G_{2}$-structure. The short time existence for (\ref{3.1}) was proved
by Bryant and Xu \cite{Bryant-Xu2011}, see also Theorem \ref{t1.1}.

A criterion for the long time existence for the Lapalcian flow on compact manifolds was
given in Theorem \ref{t1.2}. In this section, we give a new elementary
proof of Lotay-Wei's result in compact case.


\subsection{Basic theory of closed $G_{2}$-structures}
\label{subsection3.1}

Let $\wedge^{3}_{+,\bullet}(\mathcal{M})\subset\wedge^{3}_{+}
(\mathcal{M},\varphi)$ be the set of all closed $G_{2}$-structures
on $\mathcal{M}$. If $\varphi\in\wedge^{3}_{+,\bullet}(\mathcal{M})$ is closed, i.e., $d\varphi=0$, then $\tau_{0}, \tau_{1}, \tau_{3}$ are all zero, so the only nonzero torsion form is
\begin{equation}
\boldsymbol{\tau}\equiv\tau_{2}=\frac{1}{2}(\tau_{2})_{ij}dx^{ij}
=\frac{1}{2}\boldsymbol{\tau}_{ij}
dx^{ij}.\label{3.2}
\end{equation}
According to (\ref{2.20}) and (\ref{2.25}), we have $\bT_{ij}=-\frac{1}{2}\boldsymbol{\tau}_{ij}$ so that
\begin{equation}
\bT\equiv\frac{1}{2}\bT_{ij}dx^{ij} \ \ \ \text{or equivalently} \ \ \ \bT=-\frac{1}{2}\boldsymbol{\tau},\label{3.3}
\end{equation}
is a $2$-form. Since $d\psi=\boldsymbol{\tau}\wedge\varphi=-\ast_{\varphi}
\boldsymbol{\tau}$, we get $d^{\ast}_{\varphi}\boldsymbol{\tau}=\ast_{\varphi}d\ast_{\varphi}
\boldsymbol{\tau}=-\ast_{\varphi}d^{2}\psi=0$ which is given in local coordinates by
\begin{equation}
\nabla^{i}\boldsymbol{\tau}_{ij}=0\label{3.4}
\end{equation}

For a closed $G_{2}$-structure $\varphi$, according to (\ref{2.23}), the Ricci curvature is given by (in this case $\bT_{ij}$ is a $2$-form)
\begin{equation*}
R_{jk}=\left(\nabla_{j}\bT_{im}-\nabla_{i}\bT_{jm}\right)\varphi^{m}{}_{k}{}^{i}
-\bT_{j}{}^{i}\bT_{ik}+\bT_{jb}\bT_{ia}\psi^{iab}{}_{k}.
\end{equation*}
Since $\boldsymbol{\tau}\in\wedge^{2}_{14}(\mathcal{M},\varphi)$ and $\bT_{ij}=-\frac{1}{2}\boldsymbol{\tau}_{ij}$, it follows from \cite{LW2017} (see page 179 -- 180) that
\begin{equation}
(\nabla_{j}\bT_{im})\varphi^{m}{}_{k}{}^{i}
=2\bT_{j}{}^{\ell}\bT_{\ell k}.\label{3.5}
\end{equation}
and therefore, for a closed $G_{2}$-structure $\varphi$, the Ricci curvature is given by
\begin{equation}
R_{jk}=-(\nabla_{i}\bT_{jm})\varphi_{k}{}^{im}
-\bT_{j}{}^{i}\bT_{ik}.\label{3.6}
\end{equation}
Taking the trace of (\ref{3.6}) yields (\ref{2.26}). Moreover, the factor $\nabla_{i}\bT_{jm}$ in (\ref{3.6}) can be
expressed as (see Proposition 2.4 in \cite{LW2017})
\begin{eqnarray}
\nabla_{i}\bT_{jk}&=&-\frac{1}{4}R_{ijmn}\varphi_{k}{}^{mn}
-\frac{1}{4}R_{kjmn}\varphi_{i}{}^{mn}+\frac{1}{4}R_{ikmn}\varphi_{j}{}^{mn}\nonumber\\
&&- \ \frac{1}{2}\bT_{im}\bT_{jn}\varphi_{k}{}^{mn}
-\frac{1}{2}\bT_{km}\bT_{jn}\varphi_{i}{}^{mn}
+\frac{1}{2}\bT_{im}\bT_{kn}\varphi_{j}{}^{mn}.\label{3.7}
\end{eqnarray}

If $\varphi$ is a closed $G_{2}$-structure, Section 2.2 in \cite{LW2017}
shows that $\pi^{3}_{7}(\Delta_{\varphi}\varphi)=0$ and hence, according
to (\ref{2.10}),
\begin{equation}
\Delta_{\varphi}\varphi=\textsf{i}_{\varphi}(h)\in\wedge^{3}_{1}(\mathcal{M},\varphi)
\oplus\wedge^{3}_{27}(\mathcal{M},\varphi),\label{3.8}
\end{equation}
where
\begin{equation}
h_{ij}=\frac{1}{2}\nabla_{m}\boldsymbol{\tau}_{ni}\varphi_{j}{}^{mn}
-\frac{1}{6}|\boldsymbol{\tau}|^{2}_{\varphi}g_{ij}-\frac{1}{4}\boldsymbol{\tau}_{i}{}^{\ell}
\boldsymbol{\tau}_{\ell j}=-R_{ij}-\frac{2}{3}|\bT|^{2}_{\varphi}g_{ij}
-2\bT_{i}{}^{k}\bT_{kj}.\label{3.9}
\end{equation}
Here $|\bT|^{2}_{\varphi}=\frac{1}{2}\bT_{k\ell}\bT^{k\ell}
=\frac{1}{2}||\bT||^{2}_{\varphi}$.

\subsection{Evolution equations for closed $G_{2}$-structures}
\label{subsection3.2}

Since the Laplacian flow (\ref{3.1}) preserves the closedness of $\varphi_{t}$, it follows from (\ref{3.10}) that we have
\begin{equation}
\Delta_{\varphi_{t}}\varphi_{t}=\textsf{i}_{\varphi_{t}}(h_{t})
\in\wedge^{3}_{1}(\mathcal{M},
\varphi_{t})\oplus\wedge^{3}_{27}(\mathcal{M},\varphi_{t}),
\label{3.10}
\end{equation}
where
\begin{equation}
h_{ij}=-R_{ij}-\frac{2}{3}|\bT_{t}|^{2}_{t}g_{ij}-2\bT_{i}{}^{k}\bT_{kj}.
\label{3.11}
\end{equation}
From Theorem \ref{t2.3}, we see that the associated metric tensor
$g_{t}$ evolves by
\begin{equation}
\partial_{t}g_{ij}=2h_{ij}=-2R_{ij}
-\frac{4}{3}|\bT_{t}|^{2}_{t}g_{ij}-4\bT_{i}{}^{k}\bT_{kj}.\label{3.12}
\end{equation}
and the volume form $dV_{t}$ evolves by
\begin{equation}
\partial_{t}dV_{t}=({\rm tr}_{t}h_{t})dV_{t}=
\left(-R-\frac{14}{3}|\bT_{t}|^{2}_{t}+4|\bT_{t}|^{2}_{t}\right)
dV_{t}
=\frac{4}{3}|\bT_{t}|^{2}_{t}dV_{t}.\label{3.13}
\end{equation}
Hence, along the flow (\ref{3.1}), the volume of $g_{t}$ is
nondecreasing.
\\

Introduce the following notions
\begin{equation}
\blacksquare_{t}:=\partial_{t}-\blacktriangle_{t}, \ \ \ |\cdot|_{t}:=|\cdot|_{\varphi_{t}}, \ \ \ \Delta_{t}:=\Delta_{\varphi_{t}},\label{3.14}
\end{equation}
where $\blacktriangle_{t}:=g^{ij}\nabla_{i}\nabla_{j}$ is the usual Laplacian of
$g_{t}$ and $\Delta_{t}$ is the Hodge Laplacian of $g_{t}$, and also the $2$-tenor ${\rm Sic}_{t}$ with components
\begin{equation}
S_{ij}:=R_{ij}+\frac{2}{3}|\bT_{t}|^{2}_{t}g_{ij}
+2\bT_{i}{}^{k}\bT_{kj}=-h_{ij}.\label{3.15}
\end{equation}
Then the evolution equation (\ref{3.12}) can be written as
\begin{equation}
\partial_{t}g_{ij}=-2S_{ij}.\label{3.16}
\end{equation}
Moreover, the trace of ${\rm Sic}_{t}$ is exactly the scalar
curvature, up to a multiplying constant,
\begin{equation}
S_{t}:={\rm tr}_{t}{\rm Sic}_{t}=R_{t}+\frac{14}{3}|\bT_{t}|^{2}_{t}
-4|\bT_{t}|^{2}_{t}=-\frac{4}{3}|\bT_{t}|^{2}_{t}=\frac{2}{3}R_{t}.
\label{3.17}
\end{equation}
It was proved in \cite{LW2017} that
\begin{equation}
|\Delta_{t}\varphi_{t}|^{2}_{t}
=({\rm tr}_{t}h_{t})^{2}+2||h_{t}||^{2}_{t}
=\frac{16}{9}|\bT_{t}|^{4}_{t}+2||{\rm Sic}_{t}||^{2}_{t}.\label{3.18}
\end{equation}
This identity together with (\ref{2.26}) shows that the boundedness
of $\Delta_{t}\varphi_{t}$ is equivalent to the boundedness
of ${\rm Ric}_{t}$.
\\

The evolution equation (\ref{2.33}) implies that for the Laplacian
flow on closed $G_{2}$-structures, the torsion $T_{ij}$ evolves by
evolves
\begin{equation}
\partial_{t}\bT_{ij}=\bT_{i}{}^{k}h_{kj}
-(\nabla_{m}h_{ni})\varphi_{j}{}^{mn}.\label{3.19}
\end{equation}
Furthermore, we can prove

\begin{proposition}\label{p3.1} Under the flow (\ref{3.1}), we have
\begin{eqnarray}
\blacksquare_{t}\bT_{ij}&=&3R_{j}{}^{k}\bT_{ki}-R_{i}{}^{k}\bT_{kj}
-\frac{1}{2}R_{ijmk}\bT^{mk}-\frac{1}{2}R_{mpi}{}^{k}R_{qk}\psi_{j}{}^{pqm}-
\frac{2}{3}\varphi_{ji}{}^{m}
\nabla_{m}|\bT_{t}|^{2}_{t}\nonumber\\
&&+ \ \nabla_{p}\bT_{qi}\left(\bT^{pk}\varphi_{kj}{}^{q}
-2\bT^{qk}\varphi_{kj}{}^{p}\right)-\frac{2}{3}|\bT_{t}|^{2}_{t}
\bT_{ij}
-4T_{i}{}^{k}\bT_{k}{}^{m}\bT_{mj}.\label{3.20}
\end{eqnarray}
\end{proposition}

\begin{proof} See \cite{LW2017}.
\end{proof}

For a geometric flow $\partial_{t}g_{ij}=\eta_{ij}$, for some symmetric $2$-tensor $\eta_{ij}$, we have
\begin{eqnarray*}
\partial R^{\ell}_{ijk}
&=&\frac{1}{2}g^{\ell p}\bigg(\nabla_{i}\nabla_{j}\eta_{kp}
+\nabla_{i}\nabla_{k}\eta_{jp}-\nabla_{i}\nabla_{p}\eta_{jk}\\
&&- \ \nabla_{j}\nabla_{i}\eta_{kp}-\nabla_{j}\nabla_{k}
\eta_{ip}+\nabla_{j}\nabla_{p}\eta_{ik}\bigg),\\
\partial_{t}R_{jk}&=&\frac{1}{2}g^{pq}
\left(\nabla_{q}\nabla_{j}\eta_{kp}+\nabla_{q}\nabla_{k}\eta_{jp}
-\nabla_{q}\nabla_{p}\eta_{jk}-\nabla_{j}\nabla_{k}\eta_{qp}\right),\\
\partial_{t}R_{t}&=&-\blacktriangle_{t}{\rm tr}_{t}\eta_{t}
+{\rm div}_{t}({\rm div}_{t}\eta_{t})
-R_{ij}h^{ij},
\end{eqnarray*}
where $({\rm div}_{t}\eta_{t})_{j}=\nabla^{i}\eta_{ij}$. Applying those evolution equations to $\eta_{ij}=-2R_{ij}-\frac{4}{3}|\bT_{t}|^{2}_{t}
g_{ij}-4\bT_{i}{}^{k}\bT_{kj}=-2S_{ij}$ we have
\begin{eqnarray*}
{\rm tr}_{t}\eta_{t}&=&-2R_{t}-\frac{28}{3}|\bT_{t}|^{2}_{t}
+8|\bT_{t}|^{2}_{t} \ \ = \ \ \frac{8}{3}|\bT_{t}|^{2}_{t},\\
({\rm div}_{t}\eta_{t})_{j}&=&-2\nabla^{i}R_{ij}-\frac{4}{3}\nabla_{j}|\bT_{t}|^{2}_{t}
-4\nabla^{i}\widehat{\bT}_{ij}\\
&=&-\nabla_{j}R_{t}-\frac{4}{3}\nabla_{j}|\bT_{t}|^{2}_{t}-4\nabla^{i}
\widehat{\bT}_{ij},\\
{\rm div}_{t}({\rm div}_{t}\eta_{t})&=&\nabla^{j}({\rm div}_{t}\eta_{t})_{j} \ \ = \ \ -\blacktriangle_{t}R_{t}-\frac{4}{3}\blacktriangle_{t}|\bT_{t}|^{2}_{t}
-4\nabla^{j}\nabla^{i}\widehat{\bT}_{ij},
\end{eqnarray*}
where the symmetric $2$-tensor $\widehat{\bT}$ is given by
\begin{equation}
\widehat{\bT}_{ij}:=\bT_{ik}\bT^{k}{}_{j}.\label{3.21}
\end{equation}
Plugging those identities into the above evolution equation for $R_{t}$,
we get
\begin{eqnarray*}
\partial_{t}R_{t}&=&-4\blacktriangle_{t}|\bT_{t}|^{2}_{t}-\blacktriangle_{t}R_{t}
-4\nabla^{j}\nabla^{i}\widehat{\bT}_{ij}-R^{ij}\left(-2R_{ij}-
\frac{4}{3}|\bT_{t}|^{2}_{t}g_{ij}
-4\widehat{\bT}_{ij}\right)\\
&=&\blacktriangle_{t}R_{t}-4\nabla^{j}\nabla^{i}\widehat{\bT}_{ij}
+2||{\rm Ric}_{t}||^{2}_{t}+\frac{4}{3}|\bT_{t}|^{2}_{t}R_{t}
+4R^{ij}\widehat{\bT}_{ij}
\end{eqnarray*}
which implies
\begin{equation}
\blacksquare_{t}R_{t}=2||{\rm Ric}_{t}||^{2}_{t}
-\frac{2}{3}R^{2}_{t}-4\nabla^{j}\nabla^{i}\widehat{\bT}_{ij}+4\langle\langle{\rm Ric}_{t},\widehat{\bT}\rangle\rangle_{t}.\label{3.22}
\end{equation}
Observe that the last two terms on the right-hand side of (\ref{3.22})
are not determined of their signs. In the following, we shall use
the identity
\begin{equation}
\nabla^{i}\bT_{ij}
=0\label{3.23}
\end{equation}
follows from from (\ref{3.3}) and (\ref{3.4}), to simplify those
two terms. Using the identity (\ref{3.23}), the term $\nabla^{j}\nabla^{i}\widehat{\bT}_{ij}$ can be simplified as follows.
\begin{eqnarray*}
\nabla^{j}\nabla^{i}\widehat{\bT}_{ij}&=&
\nabla^{j}\nabla^{i}\left(\bT_{i}{}^{k}\bT_{kj}\right) \ \ = \ \
\nabla^{j}\left[(\nabla^{i}\bT_{i}{}^{k})\bT_{kj}+\bT_{i}{}^{k}
(\nabla^{i}\bT_{kj})
\right]\\
&=&\bT^{ik}(\nabla_{j}\nabla_{i}\bT_{k}{}^{j})-(\nabla^{j}\bT^{ik})
(\nabla_{i}\bT_{jk}).
\end{eqnarray*}
On the other hand, from the Ricci identity \begin{equation*}
\nabla_{j}\nabla_{i}\bT_{k}{}^{j}=\nabla_{i}\nabla_{j}\bT_{k}{}^{j}
-R_{jik\ell}\bT^{\ell j}-R_{ji}{}^{j\ell}\bT_{k\ell}
=R_{ijk\ell}\bT^{\ell j}+R_{i\ell}\bT_{k}{}^{\ell},
\end{equation*}
we see that the evolution equation (\ref{3.22}) is equivalent to
\begin{equation}
\blacksquare_{t}R_{t}=2||{\rm Ric}_{t}||^{2}_{t}
-\frac{2}{3}R^{2}_{t}+4R_{ijk\ell}\bT^{ik}\bT^{j\ell}+4(\nabla^{j}\bT^{ik})
(\nabla_{i}\bT_{jk}).\label{3.24}
\end{equation}
From (\ref{3.15}) and (\ref{3.21}) we can rewrite the term $||{\rm
Ric}_{t}||^{2}_{t}$ in (\ref{3.24}) in terms of ${\rm Sic}_{t}$ according to the following relation:
$$
||{\rm Sic}_{t}||^{2}_{t} \ \ = \ \ \left(R_{ij}+\frac{2}{3}|\bT_{t}|^{2}_{t}g_{ij}
+2\widehat{\bT}_{ij}\right)\left(R^{ij}+\frac{2}{3}|\bT_{t}|^{2}_{t}g^{ij}
+2\widehat{\bT}^{ij}\right)
$$
$$
= \ \ ||{\rm Ric}_{t}||^{2}_{t}+\frac{4}{3}|\bT_{t}|^{2}_{t}R_{t}
+4\langle\langle{\rm Ric}_{t},\widehat{\bT}_{t}\rangle\rangle_{t}
+\frac{28}{9}|\bT_{t}|^{4}_{t}+\frac{8}{3}|\bT_{t}|^{2}_{t}
{\rm tr}_{t}\widehat{\bT}_{t}+4||\widehat{\bT}_{t}||^{2}_{t}
$$
$$
= \ \ ||{\rm Ric}_{t}||^{2}_{t}-\frac{2}{3}R^{2}_{t}
+4\langle\langle{\rm Ric}_{t}, \widehat{\bT}_{t}\rangle\rangle_{t}
+\frac{7}{9}R^{2}_{t}-\frac{4}{3}R^{2}_{t}+4||\widehat{\bT}_{t}||^{
2}_{t}
$$
$$
= \ \ ||{\rm Ric}_{t}||^{2}_{t}+4||\widehat{\bT}_{t}||^{2}_{t}
+4\langle\langle{\rm Ric}_{t},\widehat{\bT}_{t}\rangle\rangle_{t}
-\frac{11}{9}R^{2}_{t},
$$
where we used ${\rm tr}_{t}\widehat{\bT}_{t}=g^{ij}\bT_{ik}\bT^{k}{}_{j}
=\bT_{ik}\bT^{ki}=-2|\bT_{t}|^{2}_{t}$ and $R_{t}=-2|\bT_{t}|^{2}_{t}$.
Replacing $R_{t}$ by $S_{t}$ according to the identity (\ref{3.17}),
we can rewrite (\ref{3.24}) as
\begin{eqnarray*}
\blacksquare_{t}S_{t}&=&\frac{4}{3}||{\rm Sic}_{t}||^{2}_{t}
-\frac{16}{3}||\widehat{\bT}||^{2}_{t}-\frac{16}{3}\langle\langle{\rm Ric}_{t},
\widehat{\bT}_{t}\rangle\rangle_{t}+\frac{32}{27}R^{2}_{t}\\
&&+ \ \frac{8}{3}R_{ijk\ell}\bT^{ik}\bT^{j\ell}
+\frac{8}{3}(\nabla^{j}\bT^{ik})(\nabla_{i}\bT_{jk}).
\end{eqnarray*}
Similarly, replacing $\langle\langle{\rm Ric}_{t},\widehat{\bT}_{t}\rangle
\rangle_{t}$ by $\langle\langle{\rm Sic}_{t}, \widehat{\bT}_{t}\rangle
\rangle_{t}$ with respect to the identity
\begin{equation*}
\langle\langle{\rm Sic}_{t},\widehat{\bT}_{t}\rangle\rangle
=\left(R_{ij}+\frac{2}{3}|\bT_{t}|^{2}_{t}g_{ij}+2\widehat{\bT}_{ij}
\right)\widehat{\bT}^{ij}=\langle\langle{\rm Ric}_{t},\widehat{\bT}_{t}\rangle\rangle_{t}
-\frac{1}{3}R^{2}_{t}+2||\widehat{\bT}_{t}||^{2}_{t},
\end{equation*}
we obtain the following evolution equation for $S_{t}$,
\begin{equation}
\blacksquare_{t}S_{t}=\frac{4}{3}\left[\left|\left|{\rm Sic}_{t}-2\widehat{\bT}_{t}
\right|\right|^{2}_{t}-S^{2}_{t}\right]
+\frac{8}{3}\left[R_{ijk\ell}\bT^{ik}\bT^{j\ell}
+(\nabla^{j}\bT^{ik})(\nabla_{i}\bT_{jk})\right].\label{3.25}
\end{equation}
Next, we try to deal with the last bracket in (\ref{3.25}), which contains two terms $R_{ijk\ell}\bT^{ik}\bT^{j\ell}$ and $(\nabla^{j}\bT^{ik})(\nabla_{i}
\bT_{jk})$. Using (\ref{2.27}) and (\ref{3.7}), the term $(\nabla^{j}\bT^{ik})(\nabla_{i}
\bT_{jk})$ is equal to
$$
(\nabla^{j}\bT^{ik})(\nabla_{i}\bT_{jk}) \ \ = \ \
\left[\nabla^{i}\bT^{jk}+\left(\frac{1}{2}R^{ij}{}_{ab}
+\bT^{i}{}_{a}\bT^{j}{}_{b}\right)\varphi^{kab}\right]
\nabla_{i}\bT_{jk}
$$
$$
= \ \ ||\nabla_{t}\bT_{t}||^{2}_{t}
+\frac{1}{2}\left(\frac{1}{2}R^{ij}{}_{ab}+\bT^{i}{}_{a}\bT^{j}{}_{b}\right)\bigg[
-\frac{1}{2}R_{ijmn}\varphi^{mn}{}_{k}\varphi^{kab}
-\frac{1}{2}R_{kjmn}\varphi_{i}{}^{mn}\varphi^{kab}
$$
$$
+ \ \frac{1}{2}R_{ikmn}\varphi_{j}{}^{mn}\varphi^{kab}
-\bT_{im}\bT_{jn}\varphi^{mn}{}_{k}\varphi^{kab}
-\bT_{km}\bT_{jn}\varphi_{i}{}^{mn}\varphi^{kab}
+\bT_{im}\bT_{kn}\varphi_{j}{}^{mn}\varphi^{kab}\bigg].
$$
By symmetry the term
\begin{equation*}
\left(\frac{1}{2}R^{ij}{}_{ab}+\bT^{i}{}_{a}\bT^{j}{}_{b}\right)
\left(-\frac{1}{2}R_{kjmn}\varphi_{i}{}^{mn}\varphi^{kab}+
\frac{1}{2}R_{ikmn}\varphi_{j}{}^{mn}\varphi^{kab}\right)
\end{equation*}
is equal to, interchanging $i\leftrightarrow j$ and $a\leftrightarrow b$
in the second term,
\begin{equation*}
\left(\frac{1}{2}R^{ij}{}_{ab}+\bT^{i}{}_{a}\bT^{j}{}_{b}\right)
\left(-\frac{1}{2}R_{kjmn}\varphi_{i}{}^{mn}\varphi^{kab}\right)
+\left(\frac{1}{2}R^{ji}{}_{ba}+\bT^{j}{}_{b}\bT^{i}{}_{a}\right)
\left(\frac{1}{2}R_{jkmn}\varphi_{i}{}^{mn}\varphi^{kba}\right)
\end{equation*}
which is zero. Similarly, we have, by interchanging $m\leftrightarrow n$
and then $i\leftrightarrow j$, $a\leftrightarrow b$ in the first term,
$$
\left(\frac{1}{2}R^{ij}{}_{ab}+\bT^{i}{}_{a}\bT^{j}{}_{b}\right)
\left(-\bT_{km}\bT_{jn}\varphi_{i}{}^{mn}\varphi^{kab}
+\bT_{im}\bT_{kn}\varphi_{j}{}^{mn}\varphi^{kab}\right)
$$
$$
= \ \ \left(\frac{1}{2}R^{ij}{}_{ab}+\bT^{i}{}_{a}\bT^{j}{}_{b}\right)
\left(-\bT_{kn}\bT_{jm}\varphi_{i}{}^{nm}\varphi^{kab}
+\bT_{im}\bT_{kn}\varphi_{j}{}^{mn}\varphi^{kab}\right)
$$
$$
= \ \ \left(\frac{1}{2}R^{ij}{}_{ab}+\bT^{i}{}_{a}\bT^{j}{}_{b}\right)
\left(-\bT_{kn}\bT_{im}\varphi_{j}{}^{nm}\varphi^{kba}
+\bT_{im}\bT_{kn}\varphi_{j}{}^{mn}\varphi^{kab}\right) \ \ = \ \ 0.
$$
Therefore, using the identity $\varphi_{ijk}\varphi^{k}{}_{ab}
=g_{ia}g_{jb}-g_{ib}g_{ja}+\psi_{ijab}$ (see \cite{Karigiannis2007}), we arrive at
$$
(\nabla^{j}\bT^{ik})(\nabla_{i}\bT_{jk}) \ = \ ||\nabla_{t}\bT_{t}||^{2}_{t}-\frac{1}{2}\left(\frac{1}{2}R^{ij}{}_{ab}
+\bT^{i}{}_{a}\bT^{j}{}_{b}
\right)\left(\frac{1}{2}R_{ij}{}^{mn}+\bT_{i}{}^{m}\bT_{j}{}^{n}
\right)\varphi_{mnk}\varphi^{kab}
$$
$$
= \ \ ||\nabla_{t}\bT_{t}||^{2}_{t}-\frac{1}{2}\left(\frac{1}{2}R^{ij}{}_{ab}
+\bT^{i}{}_{a}\bT^{j}{}_{b}\right)
\left(\frac{1}{2}R_{ij}{}^{mn}+\bT_{i}{}^{m}\bT_{j}{}^{n}\right)\left(\delta^{a}_{m}\delta^{b}_{n}-\delta^{b}_{m}\delta^{a}_{n}
+\psi_{mn}{}^{ab}\right)
$$
$$
= \ \ ||\nabla_{t}\bT_{t}||^{2}_{t}
-\frac{1}{8}\left(R_{ijab}+2\bT_{ia}\bT_{jb}
\right)\bigg[\left(R^{ijab}+2\bT^{ia}\bT^{jb}\right)
$$
$$
- \ \left(R^{ijba}+2\bT^{ib}\bT^{ja}\right)
+\left(R^{ijmn}+2\bT^{im}\bT^{jn}\right)
\psi_{mn}{}^{ab}\bigg].
$$
Since, by our convention,
\begin{eqnarray*}
\left(R_{ijab}+2\bT_{ia}\bT_{jb}\right)
\left(R^{ijab}+2\bT^{ia}\bT^{jb}\right)
&=&||{\rm Rm}_{t}||^{2}_{t}+4R_{ijab}\bT^{ia}\bT^{jb}
+4||\bT_{t}||^{4}_{t},\\
\left(R_{ijab}+2\bT_{ia}\bT_{jb}\right)
\left(R^{ijba}+2\bT^{ib}\bT^{ja}\right)&=&
-||{\rm Rm}_{t}||^{2}_{t}-4R_{ijab}\bT^{ia}\bT^{jb}
+4||\widehat{\bT}_{t}||^{2}_{t},
\end{eqnarray*}
it follows that
\begin{eqnarray*}
(\nabla^{j}\bT^{ik})(\nabla_{i}\bT_{jk})
&=&||\nabla_{t}\bT_{t}||^{2}_{t}
+\frac{1}{8}\bigg[-2||{\rm Rm}_{t}||^{2}_{t}
-8R_{ijab}\bT^{ia}\bT^{jb}-4||\bT_{t}||^{4}_{t}\\
&&+ \ 4||\widehat{\bT}_{t}||^{2}_{t}-\left(R_{ijab}+2\bT_{ia}\bT_{jb}\right)
\left(R^{ijmn}+2\bT^{im}\bT^{jn}\right)\psi_{mn}{}^{ab}\bigg]
\end{eqnarray*}
and (\ref{3.25}) can be written as
\begin{eqnarray}
\blacksquare_{t}S_{t}&=&\frac{4}{3}\left|\left|{\rm Sic}_{t}
-2\widehat{\bT}_{t}\right|\right|^{2}_{t}+\frac{8}{3}||\nabla_{t}\bT_{t}
||^{2}_{t}+\frac{4}{3}||\widehat{\bT}_{t}||^{2}_{t}-\frac{2}{3}
||{\rm Rm}_{t}||^{2}_{t}-\frac{13}{3}S^{2}_{t}\nonumber\\
&&- \ \frac{1}{3}\left(R_{ijab}+2\bT_{ia}\bT_{jb}\right)
\left(R^{ijmn}+2\bT^{im}\bT^{jn}\right)\psi_{mn}{}^{ab}.\label{3.26}
\end{eqnarray}
Finally, we deal with the last term $J$ on the right-hand side of
(\ref{3.26}). From the identity $\psi_{ijk\ell}
\psi^{ijk\ell}=168$, we find that
\begin{eqnarray*}
J&:=&-\frac{1}{3}\left(R_{ijab}+2\bT_{ia}\bT_{jb}\right)
\left(R^{ijmn}+2\bT^{im}\bT^{jn}\right)\psi_{mn}{}^{ab}\\
&=&\frac{1}{3}\left(-R_{ij}{}^{ab}R^{ijmn}\psi_{mnab}-4\bT_{i}{}^{a}
\bT_{j}{}^{b}R^{ijmn}\psi_{mnab}
-4\bT^{a}{}_{i}\bT^{im}\bT^{b}{}_{j}\bT^{jn}\psi_{mnab}\right)\\
&=&\frac{1}{3}\bigg[\left|\left|R_{ij}{}^{ab}R^{ijmn}-\frac{1}{2}
\psi^{abmn}\right|\right|^{2}_{t}-\left|\left|R_{ij}{}^{ab}R^{ijmn}
\right|\right|^{2}_{t}
-\frac{168}{4}\\
&&+ \ \left|\left|2\bT_{i}{}^{a}\bT_{j}{}^{b}R^{ijmn}-\psi^{abmn}\right|\right|^{2}_{t}
-4\left|\left|\bT_{i}{}^{a}\bT_{j}{}^{b}R^{ijmn}\right|\right|^{2}_{t}
-168\\
&&+ \ \left|\left|2\widehat{\bT}^{am}\widehat{\bT}^{bn}-\psi^{mnab}
\right|\right|^{2}_{t}-4||\widehat{\bT}_{t}||^{4}_{t}-168\bigg].
\end{eqnarray*}
Plugging the expression for $J$ into (\ref{3.26}), we obtain

\begin{proposition}\label{p3.2} The scalar curvature $R_{t}$ or $S_{t}$ evolves by
\begin{eqnarray}
\blacksquare_{t}S_{t}&=&
\frac{4}{3}\left|\left|{\rm Sic}_{t}-2\widehat{\bT}_{t}\right|\right|^{2}_{t}
+\frac{8}{3}||\nabla_{t}\bT||^{2}_{t}+\frac{1}{3}
\left|\left|R_{ijab}R^{ij}{}_{mn}-\psi_{abmn}
\right|\right|^{2}_{t}+\frac{4}{3}||\widehat{\bT}_{t}||^{2}_{t}\nonumber\\
&&+ \ \frac{1}{3}\left|\left|2\bT_{ia}\bT_{jb}R^{ij}{}_{mn}
-\psi_{abmn}\right|\right|^{2}_{t}
+\frac{1}{3}\left|\left|2\widehat{\bT}_{am}\widehat{\bT}_{bn}
-\psi_{abmn}\right|\right|^{2}_{t}-\frac{4}{3}
||\widehat{\bT}_{t}||^{4}_{t}\label{3.27}\\
&&- \ \frac{2}{3}||{\rm Rm}_{t}||^{2}_{t}
-\frac{13}{3}S^{2}_{t}-\frac{1}{3}\left|\left|R_{ijab}R^{ij}{}_{mn}
\right|\right|^{2}_{t}-\frac{4}{3}\left|\left|\bT_{ia}\bT_{jb}
R^{ij}{}_{mn}\right|\right|^{2}_{t}-126.\nonumber
\end{eqnarray}
\end{proposition}

Since $S_{t}=\frac{2}{3}R_{t}$, it follows from the above theorem that (\ref{1.6}) holds true.
\\

Before giving local curvature estimates for Laplacian flow in the next
subsection, we derive evolution equations for ${\rm Ric}_{t}$, ${\rm Rm}_{t}$, and $\bT_{t}$ in different forms. Using the Lichnerowicz Laplacian
\begin{equation*}
\blacktriangle_{L,t}\eta_{jk}:=\blacktriangle_{t}
\eta_{jk}-R_{j}{}^{p}\eta_{pk}-R_{k}{}^{p}\eta_{jp}
+2R_{pjkq}h^{qp},
\end{equation*}
we see that the evolution equation for $R_{ij}$ can be written as
\begin{equation*}
\partial_{t}R_{jk}=-\frac{1}{2}
\left[\blacktriangle_{L,t}\eta_{jk}
+\nabla_{j}\nabla_{k}{\rm tr}_{t}\eta_{t}+\nabla_{j}(d^{\ast}_{t}\eta_{t})_{k}
+\nabla_{k}(d^{\ast}_{t}\eta_{t})_{j}\right],
\end{equation*}
where $(d^{\ast}_{t}\eta_{t})_{k}:=-\nabla^{j}\eta_{jk}$. For $\eta_{ij}
=-2R_{ij}-\frac{4}{3}||\bT_{t}||^{2}_{t}g_{ij}
-4\bT_{i}{}^{k}\bT_{kj}$ we have proved ${\rm tr}_{t}\eta_{t}=\frac{8}{3}||\bT_{t}||^{2}_{t}$ and $(d^{\ast}_{t}\eta_{t})_{j}=\nabla_{j}R_{t}+\frac{4}{3}\nabla_{j}
||\bT_{t}||^{2}_{t}
+4\nabla^{i}\widehat{\bT}_{ij}$ with $\widehat{\bT}_{ij}=
\bT_{i}{}^{k}\bT_{kj}$. Then
$$
\partial_{t}R_{jk} \ \ = \ \ \blacktriangle_{L,t}
\left(R_{jk}+\frac{2}{3}||\bT_{t}||^{2}_{t}g_{jk}+2\widehat{
\bT}_{jk}\right)
-\frac{1}{2}\nabla_{j}\left(\nabla_{k}R_{t}+\frac{4}{3}\nabla_{k}
||\bT_{t}||^{2}_{t}+4\nabla^{i}\widehat{\bT}_{ik}\right)
$$
$$
- \ \frac{4}{3}\nabla_{j}\nabla_{k}||\bT_{t}||^{2}_{t}
-\frac{1}{2}\nabla_{k}
\left(\nabla_{j}R_{t}+\frac{4}{3}\nabla_{j}||\bT_{t}||^{2}_{t}
+4\nabla^{i}\widehat{\bT}_{ij}\right)
$$
$$
= \ \ \blacktriangle_{L,t}\left(R_{jk}+\frac{2}{3}||\bT_{t}||^{2}_{t}g_{jk}
+2\widehat{\bT}_{jk}\right)-2\nabla_{j}\nabla^{i}\widehat{\bT}_{ik}
-2\nabla_{k}\nabla^{i}\widehat{\bT}_{ij}-\frac{2}{3}\nabla_{j}\nabla_{k}
||\bT_{t}||^{2}_{t}.
$$
But the first term is equal to
$$
\blacktriangle_{L,t}\left(R_{jk}+\frac{2}{3}||\bT_{t}||^{2}_{t}g_{jk}
+2\widehat{\bT}_{jk}\right) \ \ = \ \ \blacktriangle_{t}R_{jk}-2R_{j}{}^{p}R_{pk}
+2R_{pjkq}R^{pq}
$$
$$+ \ \bigg[\frac{2}{3}\left(\blacktriangle_{t}||\bT_{t}||^{2}_{t}
\right)g_{jk}+2\blacktriangle_{t}\widehat{\bT}_{jk}
-2R_{j}{}^{p}\widehat{\bT}_{pk}-2\widehat{\bT}_{j}{}^{p}R^{p}{}_{k}
+4R_{pjkq}\widehat{\bT}^{pq}\bigg],
$$
we have
\begin{eqnarray}
\blacksquare_{t}R_{ij}&=&-2R_{i}{}^{p}R_{pj}
+2R_{pijq}R^{pq}+\bigg[\frac{2}{3}\left(\blacktriangle_{t}||\bT_{t}||^{2}_{t}
\right)g_{ij}+2\blacktriangle_{t}\widehat{\bT}_{ij}
-2R_{i}{}^{p}\widehat{\bT}_{pj}\nonumber\\
&&- \ 2\widehat{\bT}_{i}{}^{p}R_{pj}
+4R_{pijq}\widehat{T}^{pq}
-2\nabla_{i}\nabla^{p}\widehat{\bT}_{pj}
-2\nabla_{j}\nabla^{p}\widehat{\bT}_{pi}
-\frac{2}{3}\nabla_{i}\nabla_{j}||\bT_{t}||^{2}_{t}\bigg].\label{3.28}
\end{eqnarray}
Consequently, the norm of ${\rm Ric}_{t}$ satisfies
\begin{eqnarray}
\blacksquare_{t}||{\rm Ric}_{t}||^{2}_{t}&=&
-2||\nabla_{t}{\rm Ric}_{t}||^{2}_{t}
+4R_{kij\ell}R^{k\ell}R^{ij}+\bigg[\frac{4}{3}R_{t}\blacktriangle_{t}
||\bT_{t}||^{2}_{t}+8R^{k}{}_{ij}{}^{\ell}\widehat{\bT}_{k\ell}R^{ij}+\nonumber\\
&& \ \frac{8}{3}||{\rm Ric}_{t}||^{2}_{t}||\bT_{t}||^{2}_{t}
+4R^{ij}\blacktriangle_{t}\widehat{\bT}_{ij}
-8 R^{ij}\nabla_{i}\nabla^{k}\widehat{\bT}_{kj}
-\frac{4}{3}R^{ij}\nabla_{i}\nabla_{j}||\bT_{t}||^{2}_{t}\bigg].
\label{3.29}
\end{eqnarray}
The general formula for $R_{ijk}^{\ell}$ gives
\begin{eqnarray}
\partial_{t}R^{\ell}_{ijk}&=&-\nabla_{i}\nabla_{k}R_{j}{}^{\ell}
-\nabla_{j}\nabla^{\ell}R_{ik}+\nabla_{i}\nabla^{\ell}R_{jk}
+\nabla_{j}\nabla_{k}R_{i}{}^{\ell}+R_{ijk}{}^{q}R_{q}{}^{\ell}+R_{ij}{}^{\ell q}R_{kp}\nonumber\\
&&+ \ 2R_{ijk}{}^{q}\widehat{\bT}_{q}{}^{\ell}+2R_{ij}{}^{\ell q}\widehat{\bT}_{kp}
-\frac{2}{3}\left(\nabla_{i}\nabla_{k}||\bT_{t}||^{2}_{t}\right)g_{j}{}^{\ell}
-\frac{2}{3}\left(\nabla_{j}\nabla^{\ell}||\bT_{t}||^{2}_{t}\right)g_{ik}\nonumber\\
&&+ \ \frac{2}{3}\left(\nabla_{i}\nabla^{\ell}||\bT_{t}||^{2}_{t}\right)g_{jk}
+\frac{2}{3}\left(\nabla_{j}\nabla_{k}||\bT_{t}||^{2}_{t}
\right)g_{i}{}^{\ell}
\label{3.30}\\
&&- \ 2\nabla_{i}\nabla_{k}\widehat{\bT}_{j}{}^{\ell}-2\nabla_{j}\nabla^{\ell}\widehat{\bT}_{ik}
+2\nabla_{i}\nabla^{\ell}\widehat{\bT}_{jk}+2\nabla_{j}\nabla_{k}
\widehat{\bT}_{i}{}^{\ell}.\nonumber
\end{eqnarray}
Hence, the evolution equation for $||{\rm Rm}_{t}||^{2}_{t}$ is given by
\begin{eqnarray}
\partial_{t}||{\rm Rm}_{t}||^{2}_{t}&=&
\nabla^{2}_{t}{\rm Ric}_{t}\ast{\rm Rm}_{t}+{\rm Ric}_{t}\ast{\rm Rm}_{t}
\ast{\rm Rm}_{t}+{\rm Rm}_{t}\ast{\rm Rm}_{t}\ast\widehat{\bT}_{t}\nonumber\\
&&+ \ {\rm Ric}_{t}\ast\nabla^{2}_{t}||\bT_{t}||^{2}_{t}+{\rm Rm}_{t}\ast\nabla^{2}_{t}\widehat{\bT}_{t}
+\frac{8}{3}|\bT_{t}|^{2}_{t}||{\rm Rm}_{t}||^{2}_{t}.\label{3.31}
\end{eqnarray}
Moreover, it was proved in \cite{LW2017} that
\begin{eqnarray}
||\nabla_{t}{\rm Rm}_{t}||^{2}_{t}
&\leq&-\frac{1}{2}\blacksquare_{t}||{\rm Rm}_{t}||^{2}_{t}
+C_{1}||{\rm Rm}_{t}||^{3}_{t}\nonumber\\
&&+ \ C_{1}||{\rm Rm}_{t}||^{3/2}_{t}||\nabla^{2}_{t}\bT_{t}||_{t}
+C_{1}||{\rm Rm}_{t}||_{t}||\nabla_{t}\bT_{t}||^{2}_{t}\label{3.32}
\end{eqnarray}
where $C_{1}$ is some universal constant, and
\begin{equation}
\blacksquare_{t}\bT_{t}={\rm Rm}_{t}\ast \bT_{t}
+{\rm Rm}_{t}\ast\bT_{t}\ast\psi_{t}+\nabla_{t}\bT_{t}\ast \bT_{t}
\ast\varphi_{t}+\bT_{t}\ast\bT_{t}\ast\bT_{t}.\label{3.33}
\end{equation}
Squaring (\ref{3.33}) gives
\begin{eqnarray}
||\nabla_{t}\bT_{t}||^{2}_{t}&\leq&-\frac{1}{2}\blacksquare_{t}
||\bT_{t}||^{2}_{t}+C_{2}||{\rm Rm}_{t}||_{t}||\bT_{t}||^{2}_{t}\nonumber\\
&&+ \ C_{2}||\nabla_{t}\bT_{t}||_{t}
||\bT_{t}||^{2}_{t}+C_{2}||\bT_{t}||^{4}_{t}\label{3.34}
\end{eqnarray}
for another universal constant $C_{2}$ which may differs from $C_{1}$.
The Cauchy-Schwartz inequality shows $2C_{2}||\nabla_{t}\bT_{t}||_{t}
||\bT_{t}||^{2}_{t}\leq||\nabla_{t}\bT_{t}||^{2}
+C^{2}_{2}||\bT_{t}||^{4}_{t}$, so that the evolution inequality
(\ref{3.34}) becomes
\begin{equation}
||\nabla_{t}\bT_{t}||^{2}_{t}\leq-\blacksquare_{t}
||\bT_{t}||^{2}_{t}+C_{3}||{\rm Rm}_{t}||_{t}||\bT_{t}||^{2}_{t}
+C_{3}||\bT_{t}||^{4}_{t}.\label{3.35}
\end{equation}
Here $C_{3}$ is a universal constant.

\subsection{Local curvature estimates}\label{subsection3.3}

In this section, we consider the Laplacian flow (\ref{3.1})
on $\mathcal{M}\times[0,T]$, where $T\in(0,T_{\max})$. From now on
we always omit the time subscripts from all considered quantities. From (\ref{3.15}), (\ref{3.29}),
(\ref{3.31}), (\ref{3.32}),
and (\ref{3.35}) we have
\begin{eqnarray*}
||\nabla{\rm Ric}||^{2}&=&-\frac{1}{2}\blacksquare
||{\rm Ric}||^{2}+{\rm Ric}\ast{\rm Ric}\ast{\rm Rm}
-\frac{1}{3}\left(\blacktriangle R\right)R
-\frac{2}{3}||{\rm Ric}||^{2}R\\
&&+ \ 2\langle\langle{\rm Ric}, \blacktriangle\widehat{\bT}
\rangle\rangle+\frac{1}{3}\langle\langle{\rm Ric}, \nabla^{2}R\rangle\rangle
+{\rm Ric}\ast\widehat{\bT}\ast{\rm Rm}+{\rm Ric}\ast\nabla^{2}
\widehat{\bT},\\
||\nabla{\rm Rm}||^{2}&\leq&-\frac{1}{2}\blacksquare||{\rm Rm}||^{2}
+C||{\rm Rm}||^{3}+C||{\rm Rm}||^{3/2}
||\nabla^{2}\bT||+C||{\rm Rm}||||\nabla \bT||^{2},\\
\partial_{t}||{\rm Rm}||^{2}&=&\nabla^{2}{\rm Ric}\ast{\rm Rm}
+{\rm Ric}\ast{\rm Rm}\ast{\rm Rm}+{\rm Rm}\ast{\rm Rm}\ast\widehat{\bT}\\
&&+ \ {\rm Ric}\ast\nabla^{2}||\bT||^{2}+{\rm Rm}\ast\nabla^{2}
\widehat{\bT}+\frac{4}{3}||\bT||^{2}||{\rm Rm}||^{2},\\
||\nabla \bT||^{2}&\leq&-\blacksquare||\bT||^{2}
+C||{\rm Rm}||||\bT||^{2}+C||\bT||^{4},\\
\partial_{t}dV&=&\frac{2}{3}||\bT||^{2}dV, \ \ \ R \ \ = \ \
-||\bT||^{2}.
\end{eqnarray*}
Choose an open domain $\Omega$ of $\mathcal{M}$ and assume that
\begin{equation}
||{\rm Ric}||\leq K\label{3.36}
\end{equation}
on $\Omega\times[0,T]$,  Then the torsion $\bT$ satisfies $||\bT||\lesssim
K^{1/2}$ and metrics $g_{t}$ are all equivalent
to $g_{0}$. We also observe from (\ref{2.25}) and (\ref{3.19}) that
\begin{equation}
||{\rm Ric}||\lesssim1\Longleftrightarrow
|\Delta\varphi|\lesssim1\label{3.37}
\end{equation}
and the following simple fact
\begin{equation}
\partial_{t}||A||^{2}=\frac{p}{2}||A||^{p-2}\partial_{t}||A||^{2}
\label{3.39}
\end{equation}
for any tensor $A$.
\\

Choose a Lipschitz function $\eta$ with support in $\Omega$ and
consider the quantity
\begin{equation*}
\frac{d}{dt}\int||{\rm Rm}||^{p}\eta^{2p}dV, \ \ \
\int:=\int_{\mathcal{M}},
\end{equation*}
where $p\geq5$. As in \cite{LY2018}, we introduce the following ``good'' quantities
\begin{eqnarray*}
A_{1}&:=&\int||{\rm Rm}||^{p}\eta^{2p}dV, \ \ \ A_{2} \ \ := \ \
\int ||{\rm Rm}||^{p-1}\eta^{2p}dV,\\
A_{3}&:=&\int||{\rm Rm}||^{p-1}||\nabla
\eta||^{2}\eta^{2p-1}dV, \ \ \ A_{4} \ \ := \ \ \int||{\rm Rm}||^{p-1}
||\nabla\eta||^{2}\eta^{2p-2}dV
\end{eqnarray*}
and also ``bad'' quantities
\begin{equation*}
B_{1}:=\frac{1}{K}\int||\nabla{\rm Ric}||^{2}||{\rm Rm}||^{p-1}
\eta^{2p}dV, \ \ \ B_{2}:=
\int||\nabla{\rm Rm}||^{2}||{\rm Rm}||^{p-3}\eta^{2p}dV.
\end{equation*}
We split the proof of Theorem \ref{t1.4} into four steps.

\begin{itemize}

\item[{\bf (a)}] In the first step, we can show that, see Lemma \ref{l3.3},
\begin{eqnarray*}
\frac{d}{dt}A_{1}&\leq& B_{1}+c K B_{2}+cK A_{4}
+cK A_{1}+cK^{2}A_{2}\\
&&+ \ c\int\left(-\blacksquare||\bT||^{2}
\right)||{\rm Rm}||^{p-1}\eta^{2p}dV.
\end{eqnarray*}

\item[{\bf (b)}] In the second step, we can prove that the term
\begin{equation*}
c\int\left(-\blacksquare||\bT||^{2}
\right)||{\rm Rm}||^{p-1}\eta^{2p}dV
\end{equation*}
is bounded from above by (see (\ref{3.47}))
\begin{equation*}
B_{1}+cK B_{2}+cK^{2}A_{2}+cK A_{1}
-\frac{d}{dt}\left[\int c(-R)||{\rm Rm}||^{p-1}
\eta^{2p}dV\right].
\end{equation*}
Observe that the above integral is nonnegative, since the scalar
curvature $R$ is nonpositive along the Laplacian flow on closed
$G_{2}$-structures. Hence we obtain from the first step that, see
Lemma \ref{l3.4},
\begin{eqnarray*}
\frac{d}{dt}A_{1}&\leq&2B_{1}+cK B_{2}+cK A_{4}+cK A_{1}
+cK^{2}A_{2}\\
&&- \ \frac{d}{dt}\left[\int c(-R)||{\rm Rm}||^{p-1}
\eta^{2p}dV\right].
\end{eqnarray*}

\item[{\bf (c)}] In the next two steps, we estimate the bad terms $B_{1}$
and $B_{2}$. In the third step, $B_{1}$ is estimated by (see
(\ref{3.57}))
\begin{eqnarray*}
B_{1}&\leq& cK B_{2}+cK A_{4}+cK A_{1}
+cK^{2}A_{2}\\
&&- \
\frac{d}{dt}\left[\frac{1}{K}\int||{\rm Rm}||^{p-1}
||{\rm Ric}||^{2}\eta^{2p}dV+c\int(-R)||{\rm Rm}||^{p-1}
\eta^{2p}dV\right].
\end{eqnarray*}
Then the second step can be simplified as, see Lemma \ref{l3.5},
\begin{eqnarray*}
\frac{d}{dt}A_{1}&\leq& cK B_{2}
+cK A_{4}+cK A_{1}+cK^{2}A_{2}\\
&&- \
\frac{d}{dt}\left[\frac{1}{K}\int||{\rm Rm}||^{p-1}
||{\rm Ric}||^{2}\eta^{2p}dV+c\int(-R)||{\rm Rm}||^{p-1}
\eta^{2p}dV\right].
\end{eqnarray*}

\item[{\bf (d)}] Finally, we estimate the term $B_{2}$. In this step
we shall use the assumption that $p\geq5$. Using the inequality $||
\nabla\bT||\lesssim ||{\rm Rm}||$ and $||\nabla^{2}\bT||
\lesssim ||\nabla{\rm Rm}||+||{\rm Rm}||||\bT||
+||\nabla\bT|||\bT||+||\bT||^{3}$, we can prove (see (\ref{3.67}))
\begin{equation*}
B_{2}\leq cA_{4}+cA_{1}-
\frac{d}{dt}\left[\frac{1}{p-1}
\int||{\rm Rm}||^{p-1}\eta^{2p}dV\right].
\end{equation*}
Plugging it into the third step, we arrive at, see Lemma \ref{l3.6},
\begin{eqnarray*}
\frac{d}{dt}(A_{1}+cK A_{2})&\le&
cK(A_{1}+cK A_{2})+cK A_{4}\\
&&- \ \frac{d}{dt}\bigg[\frac{c}{K}\int||{\rm Rm}||^{p-1}
||{\rm Ric}||^{2}\eta^{2p}dV\\
&&+ \ c\int(-R)||{\rm Rm}||^{p-1}\eta^{2p}dV\bigg].
\end{eqnarray*}

\end{itemize}
If we choose a geodesic ball $\Omega:=B_{g_{0}}(
x_{0},\rho/\sqrt{K})$ and a cut-off function $\eta$ so that $
||\nabla\phi||\leq\sqrt{K}e^{cKT}/\rho$, then the above inequality
gives a proof of Theorem \ref{t1.4}.
\\

We are going to carry out the above mentioned four steps. From (\ref{3.39}) and the above evolution equations, we have
$$
\frac{d}{dt}\int||{\rm Rm}||^{p}\eta^{2p}
dV \ \ = \ \ \int
\left(\partial_{t}||{\rm Rm}||^{p}\right)
\eta^{2p}dV+\int||{\rm Rm}||^{p}\eta^{2p}
\partial_{t}dV
$$
$$
= \ \ \int
\frac{p}{2}||{\rm Rm}||^{p-2}
\left(\partial_{t}||{\rm Rm}||^{2}\right)
\eta^{2p}dV+\int
||{\rm Rm}||^{p}\eta^{2p}
\left(-\frac{2}{3}R\right)dV
$$
$$
= \ \ \int\frac{p}{2}
||{\rm Rm}||^{p-2}
\left[\begin{array}{cc}
\nabla^{2}{\rm Ric}\ast{\rm Rm}
+{\rm Ric}\ast{\rm Rm}\ast{\rm Rm}\\
+ \ {\rm Rm}
\ast{\rm Rm}\ast\widehat{\bT}+{\rm Ric}\ast\nabla^{2}||\bT||^{2}\\
+ \ {\rm Rm}\ast\nabla^{2}\widehat{\bT}
+\frac{4}{3}||\bT||^{2}||{\rm Rm}||^{2}
\end{array}\right]\eta^{2p}dV
$$
\begin{equation}
- \ \ \frac{2}{3}\int R||{\rm Rm}||^{p}
\eta^{2p}dV\label{3.39}
\end{equation}
$$
\leq \ \ c\int
||{\rm Rm}||^{p-2}\bigg[\nabla^{2}{\rm Ric}\ast{\rm Rm}
+K||{\rm Rm}||^{2}+K||{\rm Rm}||^{2}
+\nabla^{2}||\bT||^{2}\ast{\rm Ric}
$$
$$
+ \ \nabla^{2}\widehat{\bT}\ast{\rm Rm}\bigg]
\eta^{2p}dV+cK\int||{\rm Rm}||^{p}\eta^{2p}dV
$$
$$
\leq \ \ c\int
||{\rm Rm}||^{p-2}
\left[\nabla^{2}{\rm Ric}\ast{\rm Rm}
+\nabla^{2}||\bT||^{2}\ast{\rm Ric}
+\nabla^{2}\widehat{\bT}\ast{\rm Rm}\right]\eta^{2p}dV
$$
$$
+ \ cK\int||{\rm Rm}||^{p}\eta^{2p}dV.
$$
It was proved in \cite{KMW2016} that the first integral in (\ref{3.39}) is bounded by
$$
c\int||{\rm Rm}||^{p-2}
\left(\nabla^{2}{\rm Ric}\ast{\rm Rm}\right)
\eta^{2p}dV \ \ \leq \ \
\frac{1}{K}\int||\nabla{\rm Ric}||^{2}
||{\rm Rm}||^{p-1}\eta^{2p}dV
$$
\begin{equation}
+ \ cK\int
||\nabla{\rm Rm}||^{2}||{\rm Rm}||^{p-3}
\eta^{2p}dV
+cK\int||{\rm Rm}||^{p-1}||\nabla
\eta||^{2}\eta^{2p-2}dV.\label{3.40}
\end{equation}
Since $||\bT||^{2}=-R$, the same inequality holds for the integral
\begin{equation*}
c\int||{\rm Rm}||^{p-2}
\left(\nabla^{2}||\bT||^{2}\ast{\rm Ric}\right)\eta^{2p}
dV.
\end{equation*}
To deal with the last term in the bracket of (\ref{3.39}), we use the same argument of \cite{KMW2016} to conclude
$$
c\int||{\rm Rm}||^{p-2}
\left(\nabla^{2}\widehat{\bT}\ast{\rm Rm}\right)\eta^{2p}dV \ \
= \ \ c\int
\left(\nabla||{\rm Rm}||^{p-2}\ast\nabla
\widehat{\bT}\ast{\rm Rm}\right)\eta^{2p}dV
$$
$$
+ \ c\int
\left(||{\rm Rm}||^{p-2}\ast\nabla\widehat{\bT}
\ast\nabla{\rm Rm}\right)\eta^{2p}dV
$$
$$
+ \ c\int\left(||{\rm Rm}||^{p-2}\ast
\nabla\widehat{\bT}\ast{\rm Rm}\ast\nabla\eta\right)\eta^{2p-1}dV
$$
$$
\leq \ \ c\int ||{\rm Rm}||^{p-2}||\nabla{\rm Rm}||
||\nabla\widehat{\bT}||\eta^{2p}dV
+c\int||{\rm Rm}||^{p-2}||\nabla\widehat{\bT}||
||\nabla{\rm Rm}||\eta^{2p}dV
$$
$$
+ \ c\int ||{\rm Rm}||^{p-1}||\nabla\widehat{\bT}||||\nabla
\eta||\eta^{2p-1}dV
$$
$$
\leq \ \ c\int||{\rm Rm}||^{p-2}||\nabla{\rm Rm}||
||\nabla\widehat{\bT}||\eta^{2p}dV
+c\int ||{\rm Rm}||^{p-1}
||\nabla\widehat{\bT}||||\nabla\eta||\eta^{2p-1}
dV.
$$
According to the Cauchy-Schwartz inequality, the first and second
integrals are bounded by
$$
\int||{\rm Rm}||^{p-2}
||\nabla{\rm Rm}||||\nabla\widehat{\bT}||\eta^{2p}dV
$$
$$
= \ \ \int\left(||{\rm Rm}||^{\frac{p-3}{2}}||\nabla{\rm Rm}||
\eta^{p}\right)
\left(||{\rm Rm}||^{\frac{p-1}{2}}||\nabla\widehat{\bT}||
\eta^{p}\right)dV
$$
$$
\leq \ \ cK\int||\nabla{\rm Rm}||^{2}
||{\rm Rm}||^{p-3}\eta^{2p}dV
+\frac{1}{K}\int||\nabla\widehat{\bT}||^{2}
||{\rm Rm}||^{p-1}\eta^{2p}dV
$$
and
$$
\int||{\rm Rm}||^{p-1}||\nabla\widehat{\bT}||
||\nabla\eta||\eta^{2p-1}dV
$$
$$
= \ \ \int\left(||{\rm Rm}^{\frac{p-1}{2}}
||\nabla\widehat{\bT}||\eta^{p}\right)
\left(||{\rm Rm}||^{\frac{p-1}{2}}
||\nabla\eta||\eta^{p-1}\right)dV
$$
$$
\leq \ \ \frac{1}{K}\int||\nabla\widehat{\bT}||^{2}
||{\rm Rm}||^{p-1}\eta^{2p}dV
+cK\int||{\rm Rm}||^{p-1}||\nabla\eta||^{2}
\eta^{2p-2}dV.
$$
Hence we obtain
$$
c\int||{\rm Rm}||^{p-2}
\left(\nabla^{2}\widehat{\bT}
\ast{\rm Rm}\right)\eta^{2p}dV \ \
\leq \ \ \frac{1}{K}\int||\nabla\widehat{\bT}||^{2}
||{\rm Rm}||^{p-1}\eta^{2p}dV
$$
\begin{equation}
+ \ cK\int||\nabla{\rm Rm}||^{2}||{\rm Rm}||^{p-3}
\eta^{2p}dV+cK\int||{\rm Rm}||^{p-1}
||\nabla\eta||^{2}\eta^{2p-2}dV.\label{3.41}
\end{equation}
Using $\widehat{\bT}=\bT\ast\bT$ and $R=-||\bT||^{2}$ yields
$$
\frac{1}{K}\int||\nabla\widehat{\bT}||^{2}||{\rm Rm}||^{p-1}
\eta^{2p}dV
$$
\begin{equation}
\leq \ \ \frac{c}{K}\int ||\nabla\bT||^{2}
||\bT||^{2}||{\rm Rm}||^{p-1}\eta^{2p}dV \ \ \leq \ \
c\int||\nabla\bT||^{2}||{\rm Rm}||^{p-1}
\eta^{2p}dV\label{3.42}
\end{equation}
$$
\leq \ \ c\int\left(-\frac{1}{4}\blacksquare
||\bT||^{2}+c||{\rm Rm}||||\bT||^{2}+c||\bT||^{4}\right)
||{\rm Rm}||^{p-1}\eta^{2p}dV
$$
$$
= \ \ c\int\left(-\blacksquare||\bT||^{2}\right)
||{\rm Rm}||^{p-1}\eta^{2p}dV
$$
$$
+ \ cK\int||{\rm Rm}||^{p}\eta^{2p}dV
+cK^{2}\int||{\rm Rm}||^{p-1}\eta^{2p}dV.
$$

Hence, using (\ref{3.40}), (\ref{3.41}), and (\ref{3.42}), we arrive at

\begin{lemma}\label{l3.3} One has
\begin{eqnarray}
A'_{1} \ \equiv \ \frac{d}{dt}A_{1}&\leq&B_{1}+cKB_{2}
+cKA_{4}+cKA_{1}+cK^{2}A_{2}\nonumber\\
&&+ \ c\int\left(-\blacksquare||\bT||^{2}\right)
||{\rm Rm}||^{p-1}\eta^{2p}dV.\label{3.43}
\end{eqnarray}
\end{lemma}

In the following computations, we are mainly going to estimate or
simplify the bad terms $B_{1}, B_{2}$, and also the term involving $
-\blacksquare||\bT||^{2}$. Integration by parts on the last integral in (\ref{3.43}) and using
$R=-||\bT||^{2}$, we obtain
$$
c\int\left(-\blacksquare||\bT||^{2}
\right)||{\rm Rm}||^{p-1}\eta^{2p}dV \ \ = \ \
c\int\left((\partial_{t}-\Delta)R\right)||{\rm Rm}||^{p-1}
\eta^{2p}dV
$$
$$
= \ \ c\int\left(\partial_{t}R\right)
||{\rm Rm}||^{p-1}\eta^{2p}dV
+c\int\left\langle\nabla R,\nabla\left(||{\rm Rm}||^{p-1}
\eta^{2p}\right)\right\rangle dV
$$
$$
= \ \ \frac{d}{dt}\left(c\int R||{\rm Rm}||^{p-1}
\eta^{2p}dV\right)
-c\int R\left(\partial_{t}||{\rm Rm}||^{p-1}
\right)\eta^{2p}dV
$$
$$
- \ c\int R||{\rm Rm}||^{p-1}\eta^{2p}\partial_{t}dV
+c\int\left\langle\nabla R, ||{\rm Rm}||^{p-3}
{\rm Rm}\ast\nabla{\rm Rm}\right\rangle\eta^{2p}dV
$$
$$
+ \ c\int\left\langle\nabla R, ||{\rm Rm}||^{p-1}
\eta^{2p-1}\nabla\eta\right\rangle dV
$$
$$
\leq \ \ c\int||{\rm Rm}||^{p-2}
\langle\nabla R,\nabla{\rm Rm}\rangle\eta^{2p}dV
+c\int||{\rm Rm}||^{p-1}||\nabla R||||\nabla\eta||\eta^{2p-1}dV
$$
$$
+ \ c\int R^{2}||{\rm Rm}||^{p-1}\eta^{2p}dV
-c\int R\left(\partial_{t}||{\rm Rm}||^{p-1}\right)\eta^{2p}dV
$$
$$
+ \ \frac{d}{dt}\left(c\int R||{\rm Rm}||^{p-1}\eta^{2p}dV
\right).
$$
The first two integrals can be simplified by using the Cauchy-Schwarz inequality as follows:
$$
c\int||{\rm Rm}||^{p-2}\langle\nabla R,\nabla{\rm Rm}\rangle
\eta^{2p}dV \ \ \leq \ \ c\int||\nabla{\rm Ric}||||\nabla{\rm Rm}||
||{\rm Rm}||^{p-2}\eta^{2p}dV
$$
$$
\leq \ \ c\int\left(||\nabla{\rm Rm}||
||{\rm Rm}||^{\frac{p-3}{2}}\eta^{p}
\right)\left(||\nabla{\rm Ric}||||{\rm Rm}||^{\frac{p-1}{2}}
\eta^{p}\right)dV \ \ \leq \ \ \frac{1}{50}B_{1}+cK B_{2}
$$
and
$$
c\int||{\rm Rm}||^{p-1}
||\nabla R||||\nabla\eta||\eta^{2p-1}dV \ \
\leq \ \ c\int||{\rm Rm}||^{p-1}||\nabla{\rm Ric}||
||\nabla\eta||\eta^{2p-1}dV
$$
$$
\leq \ \ c\int\left(||{\rm Rm}||^{\frac{p-1}{2}}
||\nabla\eta||\eta^{p-1}\right)
\left(||{\rm Rm}||^{\frac{p-1}{2}}||\nabla{\rm Ric}||
\eta^{p}\right)dV \ \ \leq \ \ \frac{1}{50}B_{1}+cK A_{4}.
$$
Therefore
$$
c\int\left(-\blacksquare||\bT||^{2}
\right)||{\rm Rm}||^{p-1}\eta^{2p}dV \ \ \leq \ \ \frac{2}{50}B_{1}+cK B_{2}+cK A_{4}+cK^{2}A_{2}
$$
\begin{equation}
+ \ \frac{d}{dt}\left(c\int R||{\rm Rm}||^{p-1}\eta^{2p}dV
\right)-c\int R\left(\partial_{t}||{\rm Rm}||^{p-1}\right)
\eta^{2p}dV.\label{3.44}
\end{equation}
Now, the second integral in (\ref{3.44}) is equal to
$$
-c\int R\left(\partial_{t}||{\rm Rm}||^{p-1}\right)
\eta^{2p}dV \ \ = \ \ c\int\left(-R\right)
||{\rm Rm}||^{p-3}\left(\partial_{t}||{\rm Rm}||^{2}\right)
\eta^{2p}dV
$$
$$
= \ \ c\int(-R)||{\rm Rm}||^{p-3}
\bigg[\nabla^{2}{\rm Ric}\ast{\rm Rm}
+{\rm Ric}\ast{\rm Rm}\ast{\rm Rm}
+{\rm Rm}\ast{\rm Rm}\ast\widehat{\bT}
$$
$$
+ \ {\rm Ric}\ast\nabla^{2}||\bT||^{2}
+{\rm Rm}\ast\nabla^{2}\widehat{\bT}
+\frac{4}{3}||\bT||^{2}||{\rm Rm}||^{2}\bigg]\eta^{2p}dV
$$
$$
\leq \ \ c\int(-R)||{\rm Rm}||^{p-3}
\left[\nabla^{2}{\rm Ric}\ast{\rm Rm}
-{\rm Ric}\ast\nabla^{2}R+\nabla^{2}\widehat{\bT}\ast{\rm Rm}
\right]\eta^{2p}dV
+cK^{2}A_{2}.
$$
Using the identity, where $p\geq5$,
\begin{equation*}
\nabla||{\rm Rm}||^{p-3}
=\frac{p-3}{2}
\left(||{\rm Rm}||^{2}\right)^{\frac{p-3}{2}-1}
\nabla||{\rm Rm}||^{2}=||{\rm Rm}||^{p-5}
{\rm Rm}\ast\nabla{\rm Rm}
\end{equation*}
we obtain
$$
c\int(-R)||{\rm Rm}||^{p-3}\eta^{2p}(\nabla^{2}{\rm Ric}
\ast{\rm Rm})dV = c\int(-R)||{\rm Rm}||^{p-3}\eta^{2p}
(\nabla{\rm Ric}\ast\nabla{\rm Rm})dV
$$
$$
+ \ c\int\left\{\nabla\left[(-R)||{\rm Rm}||^{p-3}\phi^{2p}\right]
\ast\nabla{\rm Ric}\ast{\rm Rm}\right\}dV
$$
$$
= c\int(-R)||{\rm Rm}||^{p-3}
\eta^{2p}(\nabla{\rm Ric}\ast\nabla{\rm Rm})
dV+ c\int||{\rm Rm}||^{p-3}\eta^{2p}(\nabla R\ast\nabla{\rm Ric}
\ast{\rm Rm})dV
$$
$$
+ \ c\int(-R)\eta^{2p}\left(\nabla||{\rm Rm}||^{p-3}
\ast\nabla{\rm Ric}\ast{\rm Rm}\right)dV
$$
$$
+ \ c\int(-R)||{\rm Rm}||^{p-3}
\eta^{2p-1}\left(\nabla\phi\ast\nabla{\rm Ric}\ast{\rm Rm}
\right)dV
$$
$$
\leq \ \ c\int||{\rm Rm}||^{p-2}\eta^{2p}
||\nabla{\rm Ric}||||\nabla{\rm Rm}||dV
+c\int||\nabla{\rm Ric}||||\nabla R||||{\rm Rm}||^{p-2}\eta^{2p}dV
$$
$$
+ \ c\int||{\rm Rm}||^{p-2}||\nabla{\rm Ric}||||\nabla{\rm Rm}||
\eta^{2p}dV
+c\int||{\rm Rm}||^{p-1}\eta^{2p-1}
||\nabla\eta||||\nabla{\rm Ric}||dV
$$
$$
\leq \ \ c\int\left(||\nabla{\rm Ric}||
||{\rm Rm}||^{\frac{p-1}{2}}\eta^{p}\right)
\left(||\nabla{\rm Rm}||||{\rm Rm}||^{\frac{p-3}{2}}
\eta^{p}\right)dV
$$
$$
+ \ c\int\left(||\nabla{\rm Ric}||||{\rm Rm}||^{\frac{p-1}{2}}
\eta^{p}\right)
\left(||\nabla\phi||||{\rm Rm}||^{\frac{p-1}{2}}
\eta^{p-1}\right)dV \ \ \leq \ \ \frac{1}{50}B_{1}+cK B_{2}
+cK A_{4}.
$$
Similarly, we can prove
$$
c\int(-R)||{\rm Rm}||^{p-3}
\left(-{\rm Ric}\ast\nabla^{2} R\right)
\eta^{2p}dV
\leq\frac{1}{50}B_{1}+cK B_{2}+cKA_{4}.
$$
Using $\nabla\widehat{\bT}=\nabla \bT\ast \bT\leq c||\nabla \bT||||
\bT||
\leq c K^{1/2}||\nabla \bT||$ yields
$$
c\int(-R)||{\rm Rm}||^{p-3}\eta^{2p}
\left(\nabla^{2}\widehat{\bT}\ast{\rm Rm}\right)
dV \ \ = \ \ c\int(-R)||{\rm Rm}||^{p-3}
\eta^{2p}(\nabla\widehat{\bT}\ast\nabla{\rm Rm})dV
$$
$$
+ \ c\int\left\{\nabla
\left[(-R)||{\rm Rm}||^{p-3}\eta^{2p}\right]
\ast\nabla\widehat{\bT}\ast{\rm Rm}\right\}dV
$$
$$
= \ \ c\int(-R)||{\rm Rm}||^{p-3}\eta^{2p}
(\nabla\widehat{T}\ast\nabla{\rm Rm})
dV+c\int||{\rm Rm}||^{p-3}\eta^{2p}
(\nabla R\ast\nabla\widehat{\bT}\ast{\rm Rm})dV
$$
$$
+ \ c\int(-R)\eta^{2p}
\left(\nabla||{\rm Rm}||^{p-3}\ast\nabla\widehat{\bT}
\ast{\rm Rm}\right)dV
$$
$$
+ \ c\int(-R)||{\rm Rm}||^{p-3}\eta^{2p-1}
\left(\nabla\eta\ast\nabla\widehat{\bT}
\ast{\rm Rm}\right)dV
$$
$$
\leq \ \ c\int\left(||{\rm Rm}||^{p-2}
\eta^{2p}||\nabla{\rm Rm}||+||{\rm Rm}||^{p-1}
\eta^{2p-1}||\nabla\eta||\right)\left(K^{1/2}||\nabla \bT||\right)
dV
$$
$$
\leq \ \ c\int\left(||\nabla{\rm Rm}||
||{\rm Rm}||^{\frac{p-3}{2}}\eta\right)
\left(||\nabla \bT|| K^{1/2}||{\rm Rm}||^{\frac{p-1}{2}}
\eta^{p}\right)dV
$$
$$
+ \ \int\left(||\nabla\eta||||{\rm Rm}||^{\frac{p-1}{2}}
\eta^{p-1}\right)\left(||\nabla \bT|| K^{1/2}
||{\rm Rm}||^{\frac{p-1}{2}}\eta^{p}\right)dV
$$
$$
\leq \ \ \epsilon c\int||\nabla \bT||^{2}||{\rm Rm}||^{p-1}
\eta^{2p}dV
+\frac{cK}{\epsilon}B_{2}+\frac{cK}{\epsilon}A_{4}.
$$
According to (\ref{3.44}) we get
$$
c\int||\nabla \bT||^{2}||{\rm Rm}||^{p-1}\eta^{2p}dV
$$
$$
\leq \ \ c\int\left(-\blacksquare||\bT||^{2}\right)
||{\rm Rm}||^{p-1}\eta^{2p}dV
+cKA_{1}+cK^{2}A_{2}
$$
$$
\leq \ \ \frac{2}{50}B_{1}+cK B_{2}+cK A_{4}+cK^{2}A_{2}
+cKA_{1}
$$
$$
+ \ \frac{d}{dt}\left(c\int R||{\rm Rm}||^{p-1}
\eta^{2p}dV\right)
-c\int R\left(\partial_{t}||{\rm Rm}||^{p-1}\right)
\eta^{2p}dV
$$
$$
\leq \ \ \frac{2}{50}B_{1}+cK B_{2}
+cK A_{4}+cK^{2}A_{2}+cK A_{1}
$$
$$
+ \ \frac{d}{dt}\left(\int cR||{\rm Rm}||^{p-1}\eta^{2p}
dV\right)
+c\int(-R)||{\rm Rm}||^{p-3}\left(\partial_{t}
||{\rm Rm}||^{2}\right)\eta^{2p}dV.
$$
Hence
$$
c\int(-R)||{\rm Rm}||^{p-3}\left(\partial_{t}
||{\rm Rm}||^{2}\right)\eta^{2p}dV \ \
\leq \ \ \frac{2}{50}B_{1}+cK B_{2}
+cK A_{4}+\frac{cK}{\epsilon}B_{2}+\frac{cK}{\epsilon}A_{4}
$$
$$
+ \ \epsilon
\left[\frac{2}{50}B_{1}
+cK B_{2}+cK A_{4}+cK^{2}A_{2}+cK A_{1}
+\frac{d}{dt}\left(\int cR||{\rm Rm}||^{p-1}
\eta^{2p}dV\right)\right]
$$
$$
+ \ \epsilon c\int(-R)||{\rm Rm}||^{p-3}
\left(\partial_{t}||{\rm Rm}||^{2}\right)
\eta^{2p}dV.
$$
Choosing $\epsilon=\frac{1}{2}$ yields
$$
\frac{c}{2}\int(-R)||{\rm Rm}||^{p-3}\left(\partial_{t}||{\rm Rm}||^{2}
\right)\eta^{2p}dV
$$
$$
\leq \ \ \frac{3}{50}B_{1}+cK B_{2}
+cK A_{4}+cK^{2} A_{2}+cK A_{1}+\frac{d}{dt}
\left(\int cR||{\rm Rm}||^{p-1}\eta^{2p}dV
\right)
$$
and
$$
c\int||\nabla \bT||^{2}||{\rm Rm}||^{p-1}
\eta^{2p}dV
$$
$$
\leq \ \ \frac{8}{50}B_{1}
+cK B_{2}+c K A_{4}+c K^{2}A_{2}+c KA_{1}
+\frac{d}{dt}\left(\int 2cR||{\rm Rm}||^{p-1}\eta^{2p}dV
\right).
$$
Thus
\begin{eqnarray}
&&c\int(-R)||{\rm Rm}||^{p-3}
\left(\partial_{t}||{\rm Rm}||^{2}
\right)\eta^{2p}dV \ \ \leq \ \ \frac{3}{50}B_{1}+cK B_{2}\nonumber\\
&&+ \ cK A_{4}+c K^{2}A_{2}+c KA_{1}+\frac{d}{dt}
\left(\int cR||{\rm Rm}||^{p-1}\eta^{2p}dV
\right)\label{3.45}
\end{eqnarray}
and
\begin{eqnarray}
&& c\int||\nabla \bT||^{2}||{\rm Rm}||^{p-1}
\eta^{2p}dV \ \ \leq \ \ \frac{8}{50}B_{1}+c K B_{2}\nonumber\\
&&+ \ cK A_{4}+c K^{2}A_{2}+cK A_{1}
+\frac{d}{dt}\left(\int cR||{\rm Rm}||^{p-1}\eta^{2p}
dV\right)\label{3.46}
\end{eqnarray}
and
\begin{eqnarray}
&& c\int\left(-\blacksquare||\bT||^{2}\right)||{\rm Rm}||^{p-1}
\eta^{2p}dV \ \ \leq \ \ \frac{5}{50}B_{1}
+c K B_{2}\nonumber\\
&&+ \ c K^{2}A_{2}+cK A_{1}+\frac{d}{dt}
\left(\int cR||{\rm Rm}||^{p-1}\eta^{2p}
dV\right).\label{3.47}
\end{eqnarray}
From (\ref{3.43}) and (\ref{3.47}) we arrive at

\begin{lemma}\label{l3.4} One has
\begin{eqnarray}
A'_{1}&\leq& 2B_{1}+c K B_{2}+cK A_{4}
+c K^{2}A_{2}+c K A_{1}\nonumber\\
&&+ \ \frac{d}{dt}\left(\int cR||{\rm Rm}||^{p-1}
\eta^{2p}dV\right).\label{3.48}
\end{eqnarray}
\end{lemma}

We next estimate $B_{1}$ and $B_{2}$. Actually, we shall see that
$B_{1}$ can be estimated in terms of $B_{2}$. Hence the key step is to
estimate $B_{2}$. For $B_{1}$, using
\begin{eqnarray*}
||\nabla{\rm Ric}||^{2}
&=&-\frac{1}{2}\blacksquare||{\rm Ric}||^{2}
+{\rm Ric}\ast{\rm Ric}\ast{\rm Rm}
-\frac{1}{3}(\blacktriangle R)\bT-\frac{2}{3}R||{\rm Ric}||^{2}\\
&&+ \ 2\langle\langle{\rm Ric}, \blacktriangle\widehat{\bT}\rangle
\rangle+\frac{1}{3}\langle\langle{\rm Ric},\nabla^{2}R\rangle\rangle
+{\rm Ric}\ast\widehat{\bT}\ast{\rm Rm}+{\rm Ric}\ast\nabla^{2}
\widehat{\bT}.
\end{eqnarray*}
we obtain
$$
B_{1} \ \ \leq \ \ \frac{1}{2K}\int||{\rm Rm}||^{p-1}
\eta^{2p}
\left(\blacktriangle-\partial_{t}\right)||{\rm Ric}||^{2}dV
+c K A_{1}
$$
\begin{equation}
+ \ \frac{1}{3K}\int(-R)||{\rm Rm}||^{p-1}
\eta^{2p}\Delta R\!\ dV
+\frac{2}{K}\int\langle\langle{\rm Ric},\blacktriangle\widehat{\bT}
\rangle
\rangle||{\rm Rm}||^{p-1}\eta^{2p}dV\label{3.49}
\end{equation}
$$
+ \ \frac{1}{3K}\int\langle\langle{\rm Ric},
\nabla^{2}R\rangle\rangle||{\rm Rm}||^{p-1}
\eta^{2p}dV+\frac{1}{K}\int||{\rm Rm}||^{p-1}\left({\rm Ric}
\ast\nabla^{2}\widehat{\bT}\right)\eta^{2p}dV.
$$
From the estimates $\nabla||{\rm Ric}||^{2}
\lesssim ||{\rm Ric}||||\nabla{\rm Ric}||$, $\nabla||{\rm Rm}||^{p-1}
\lesssim ||{\rm Rm}||^{p-2}||\nabla{\rm Rm}||$, and $
\partial_{t}||{\rm Rm}||^{p-1}=\frac{p-1}{2}||{\rm Rm}||^{p-3}
\partial_{t}||{\rm Rm}||^{2}$, we have
$$
\int||{\rm Rm}||^{p-1}
\eta^{2p}
\left(\blacktriangle-\partial_{t}\right)||{\rm Ric}||^{2}dV
$$
$$
= \ \ \int\nabla||{\rm Ric}||^{2}\ast\nabla
\left(||{\rm Rm}||^{p-1}
\eta^{2p}\right)dV
-\int||{\rm Rm}||^{p-1}
\eta^{2p}\left(\partial_{t}||{\rm Ric}||^{2}
\right)dV
$$
$$
= \ \ \int\left(\nabla||{\rm Ric}||^{2}\ast\nabla
||{\rm Rm}||^{p-1}\right)
\eta^{2p}dV+\int\left(\nabla||{\rm Ric}||^{2}\ast
\nabla
\eta\right)||{\rm Rm}||^{p-1}\eta^{2p-1}dV
$$
$$
- \ \frac{d}{dt}\left[\int||{\rm Rm}||^{p-1}
\eta^{2p}
||{\rm Ric}||^{2}dV\right]
+\int\left(\partial_{t}||{\rm Rm}||^{p-1}
\right)\eta^{2p}||{\rm Ric}||^{2}dV
$$
$$
+ \ \int||{\rm Rm}||^{p-1}\eta^{2p}
||{\rm Ric}||^{2}(\partial_{t}dV)
$$
$$
\leq \ \ cK\int||\nabla{\rm Ric}||||\nabla{\rm Rm}||
||{\rm Rm}||^{p-2}\eta^{2p}dV
+cK\int||\nabla{\rm Ric}||||\nabla
\eta||||{\rm Rm}||^{p-1}\eta^{2p-1}dV
$$
$$
+ \ c\int||{\rm Rm}||^{p-3}
\left(\partial_{t}||{\rm Rm}||^{2}\right)
\eta^{2p}||{\rm Ric}||^{2}dV
+cK^{2}A_{1}
$$
$$
- \ \frac{d}{dt}\left[\int||{\rm Rm}||^{p-1}||{\rm Ric}||^{2}
\eta^{2p}dV\right]
$$
$$
\leq \ \ cK\left(\frac{1}{50c}B_{1}+c K B_{2}\right)
+cK\left(\frac{1}{50 c}B_{1}+c K A_{4}\right)
+cK^{2}A_{1}
$$
$$
+ \ c\int||{\rm Ric}||^{2}||{\rm Rm}||^{p-3}
\eta^{2p}\left(\partial_{t}
||{\rm Rm}||^{2}\right)dV
-\frac{d}{dt}\left[\int||{\rm Rm}||^{p-1}||{\rm Ric}||^{2}
\eta^{2p}dV\right]
$$
$$
\leq \ \ \frac{2}{50}K B_{1}
+cK^{2}B_{2}+c K^{2}A_{4}+cK^{2}A_{1}
$$
$$
+ \ c\int||{\rm Ric}||^{2}||{\rm Rm}||^{p-3}
\eta^{2p}\left(\partial_{t}
||{\rm Rm}||^{2}\right)dV
-\frac{d}{dt}\left[\int||{\rm Rm}||^{p-1}||{\rm Ric}||^{2}
\eta^{2p}dV\right].
$$
Thus
\begin{equation}
\int||{\rm Rm}||^{p-1}
\eta^{2p}\blacksquare||{\rm Ric}||^{2}dV \ \
\leq \ \ \frac{2}{50}KB_{1}+c K^{2}B_{2}+c K^{2}A_{4}+cK^{2}A_{1}
\label{3.50}
\end{equation}
$$
+ \ c\int||{\rm Ric}||^{2}||{\rm Rm}||^{p-3}
\eta^{2p}\left(\partial_{t}
||{\rm Rm}||^{2}\right)dV
-\frac{d}{dt}\left[\int||{\rm Rm}||^{p-1}||{\rm Ric}||^{2}
\eta^{2p}dV\right].
$$
Consider the term
$$
c\int||{\rm Ric}||^{2}||{\rm Rm}||^{p-3}
\eta^{2p}\left(\partial_{t}||{\rm Rm}||^{2}\right)
dV \ \ = \ \ c\int||{\rm Ric}||^{2}
||{\rm Rm}||^{p-3}\eta^{2p}
$$
$$
\bigg[\nabla^{2}{\rm Ric}\ast{\rm Rm}
+{\rm Ric}\ast{\rm Rm}\ast{\rm Rm}+{\rm Rm}
\ast{\rm Rm}\ast\widehat{\bT}+{\rm Ric}
\ast\nabla^{2}||\bT||^{2}+{\rm Rm}\ast\nabla^{2}
\widehat{\bT}
$$
$$
+ \ \frac{4}{3}||\bT||^{2}||{\rm Rm}||^{2}\bigg]dV \ \ \leq \ \
c\int||{\rm Ric}||^{2}
||{\rm Rm}||^{p-3}
\eta^{2p}\bigg[\nabla^{2}{\rm Ric}\ast{\rm Rm}
-\nabla^{2}R\ast{\rm Ric}
$$
$$
+ \ \nabla^{2}\widehat{\bT}\ast{\rm Rm}
\bigg]dV+cK^{2}A_{2}.
$$
The three terms in the bracket can be estimated as follows. Firstly
$$
c\int||{\rm Ric}||^{2}||{\rm Rm}||^{p-3}
\eta^{2p}\left(\nabla^{2}{\rm Ric}\ast{\rm Rm}\right)dV
$$
$$
= \ \ c\int||{\rm Ric}||^{2}||{\rm Rm}||^{p-3}\eta^{2p}
\left(\nabla{\rm Ric}\ast\nabla{\rm Rm}\right)dV
$$
$$
+ \ c\int\left\{\nabla
\left[||{\rm Ric}||^{2}||{\rm Rm}||^{p-3}
\eta^{2p}\right]\ast\nabla{\rm Ric}\ast{\rm Rm}\right\}dV
$$
$$
= \ \ c\int||{\rm Ric}||^{2}
||{\rm Rm}||^{p-3}
\eta^{2p}\left(\nabla{\rm Ric}\ast\nabla{\rm Rm}\right)
dV
$$
$$
+ \ c\int||{\rm Rm}||^{p-3}
\eta^{2p}\left(\nabla||{\rm Ric}||^{2}\ast\nabla{\rm Ric}
\ast{\rm Rm}\right)dV
$$
$$
+ \ c\int||{\rm Ric}||^{2}
\eta^{2p}\left(\nabla||{\rm Rm}||^{p-3}\ast\nabla{\rm Ric}
\ast{\rm Rm}\right)dV
$$
$$
+ \ c\int||{\rm Ric}||^{2}||{\rm Rm}||^{p-3}
\eta^{2p-1}\left(\nabla\eta\ast\nabla{\rm Ric}\ast{\rm Rm}\right)dV
$$
$$
\leq \ cK\int||{\rm Rm}||^{p-2}
\eta^{2p}||\nabla{\rm Ric}||||\nabla{\rm Rm}||dV
+cK\int||{\rm Rm}||^{p-1}\eta^{2p-1}
||\nabla{\rm Ric}||||\nabla\eta||dV
$$
$$
\leq \ \ cK\left(\epsilon B_{1}+\frac{K}{\epsilon}B_{2}
\right)+cK\left(\epsilon B_{1}+\frac{K}{\epsilon}A_{4}
\right) \ \ \leq \ \ \frac{1}{50}K B_{1}
+cK^{2}B_{2}+cK^{2}A_{4}.
$$
The same estimate holds for
\begin{equation*}
c\int||{\rm Ric}||^{2}||{\rm Rm}||^{p-3}
\eta^{2p}\left(-\nabla^{2}R\ast{\rm Ric}\right)dV.
\end{equation*}
Finally,
$$
c\int||{\rm Ric}||^{2}||{\rm Rm}||^{p-3}
\eta^{2p}\left(\nabla^{2}\widehat{\bT}\ast{\rm Rm}\right)
dV \ \ = \ \ c\int||{\rm Ric}||^{2}||{\rm Rm}||^{p-3}
\eta^{2p}
$$
$$
\left(\nabla\widehat{\bT}\ast\nabla{\rm Rm}\right)dV+c\int\left\{\nabla\left(||{\rm Ric}||^{2}
||{\rm Rm}||^{p-3}\eta^{2p}\right)\ast\nabla\widehat{\bT}
\ast{\rm Rm}\right\}dV
$$
$$
\leq \ \ c\int||{\rm Ric}||^{2}
||{\rm Rm}||^{p-3}\eta^{2p}
\left(K^{1/2}||\nabla\bT||||\nabla{\rm Rm}||\right)
dV
$$
$$
+ \ c\int\left(\nabla||{\rm Ric}||^{2}\right)
||{\rm Rm}||^{p-3}\eta^{2p}
||\nabla\widehat{\bT}||||{\rm Rm}||dV
$$
$$
+ \ c\int||{\rm Rm}||^{2}
\left(\nabla||{\rm Rm}||^{p-3}\right)
\eta^{2p}||\nabla\widehat{\bT}||||{\rm Rm}||dV
$$
$$
+ \ c\int||{\rm Ric}||^{2}||{\rm Rm}||^{p-3}
\eta^{2p-1}
||\nabla\eta||||\nabla
\widehat{\bT}||||{\rm Rm}||dV
$$
$$
\leq \ \ c K\int||{\rm Rm}||^{p-2}
\eta^{2p}\left(K^{1/2}||\nabla\bT||||\nabla{\rm Rm}||\right)dV
$$
$$
+ \ cK\int||{\rm Rm}||^{p-1}\eta^{2p-1}
\left(K^{1/2}||\nabla
\eta||||\nabla\bT||\right)dV
$$
$$
\leq \ \ K\left[cK B_{2}
+\frac{cK}{\epsilon}A_{4}
+\epsilon c\int||\nabla\bT||^{2}
||{\rm Rm}||^{p-1}\eta^{2p}dV\right]
$$
$$
\leq \ \ \frac{8}{50}K B_{1}+cK^{2}B_{2}
+cK^{2}A_{4}+cK^{3}A_{2}+cK^{2}A_{1}
+\frac{d}{dt}\left[cK\int R||{\rm Rm}||^{p-1}
\eta^{2p}
dV\right]
$$
Therefore
$$
c\int||{\rm Ric}||^{2}
||{\rm Rm}||^{p-3}
\eta^{2p}
\left(\partial_{t}||{\rm Rm}||^{2}
\right)dV \ \ \leq \ \ \frac{10}{50}K B_{1}+cK^{2}B_{2}
+cK^{2}A_{4}+cK^{3}A_{2}
$$
\begin{equation}
+ \ cK^{2}A_{1}
+cK\frac{d}{dt}\left[\int R||{\rm Rm}||^{p-1}\eta^{2p}
dV\right]\label{3.51}
\end{equation}
and
$$
\frac{1}{2K}\int||{\rm Rm}||^{p-1}
\eta^{2p}(\blacktriangle-\partial_{t})
||{\rm Ric}||^{2}dV \ \ \leq \ \
\frac{6}{50}B_{1}+cK B_{2}+cK A_{4}
+cK^{2}A_{2}+cK A_{1}
$$
$$
- \ \frac{1}{K}\frac{d}{dt}
\left[\int||{\rm Rm}||^{p-1}
||{\rm Ric}||^{2}\eta^{2p}dV
\right]
+c\frac{d}{dt}\left[\int R||{\rm Rm}||^{p-1}
\eta^{2p}dV\right]
$$
\begin{equation}
\leq \ \ \frac{6}{50}B_{1}+cKB_{2}
+cKA_{4}+cK^{2}A_{2}+cK A_{1}\label{3.52}
\end{equation}
$$
- \ \frac{d}{dt}
\left[\frac{1}{K}\int||{\rm Rm}||^{p-1}
||{\rm Ric}||^{2}\eta^{2p}dV
+c\int(-R)||{\rm Rm}||^{p-1}
\eta^{2p}dV\right].
$$
In the following, we estimate the left four terms in (\ref{3.49}). We start from terms involving the scalar curvature.
$$
\frac{1}{3K}\int(-R)||{\rm Rm}||^{p-1}\eta^{2p}
\Delta R\!\ dV \ \ = \ \ -\frac{1}{3K}
\int\nabla R\cdot\nabla\left[(-R)||{\rm Rm}||^{p-1}
\eta^{2p}\right]dV
$$
$$
= \ \ -\frac{1}{3K}\int\nabla R
\cdot\bigg[-\nabla R||{\rm Rm}||^{p-1}
\eta^{2p}+(-R)\nabla||{\rm Rm}||^{p-1}
\eta^{2p}
$$
\begin{equation}
+ \ 2p(-R)||{\rm Rm}||^{p-1}
\eta^{2p-1}\nabla\eta\bigg]dV \ \ \leq \ \ \frac{1}{3K}\int||\nabla R||^{2}
||{\rm Rm}||^{p-1}\eta^{2p}dV\label{3.53}
\end{equation}
$$
+ \ \frac{c}{K}\int(-R)||{\rm Rm}||^{p-2}||\nabla R||
||\nabla{\rm Rm}||\eta^{2p}dV
$$
$$
+ \ \frac{c}{K}\int(-R)||{\rm Rm}||^{p-1}
\eta^{2p-1}||\nabla R||||\nabla
\eta||dV
$$
$$
\leq \ \ \frac{1}{3K}
\int||\nabla R||^{2}||{\rm Rm}||^{p-1}
\eta^{2p}dV+\frac{1}{3K}\int||\nabla R||^{2}
||{\rm Rm}||^{p-1}\eta^{2p}dV
+cK B_{2}
$$
$$
+ \ \frac{1}{3K}\int||\nabla R||^{2}
||{\rm Rm}||^{p-1}\eta^{2p}dV
+cK A_{4}
$$
$$
\leq \ \ \frac{1}{K}\int||\nabla R||^{2}
||{\rm Rm}||^{p-1}\eta^{2p}dV
+cK B_{2}+cK A_{4}.
$$
The another term involving the scalar curvature can be estimated by
$$
\frac{1}{3K}\int\langle\langle {\rm Ric},
\nabla^{2}R\rangle\rangle||{\rm Rm}||^{p-1}\eta^{2p}
dV \ \ = \ \
-\frac{1}{3K}\int\nabla^{j}R\nabla^{i}\left[R_{ij}
||{\rm Rm}||^{p-1}\eta^{2p}\right]dV
$$
$$
= \ \ -\frac{1}{3K}
\int\nabla^{j}R\bigg[\frac{1}{2}\nabla_{j}R||{\rm Rm}||^{p-1}
\eta^{2p}+R_{ij}\nabla^{i}||{\rm Rm}||^{p-1}\eta^{2p}
$$
\begin{equation}
+ \ R_{ij}||{\rm Rm}||^{p-1}2p\eta^{2p-1}
\nabla^{i}\eta\bigg]dV \ \ \leq \ \
-\frac{1}{6K}\int||\nabla R||^{2}
||{\rm Rm}||^{p-1}
\eta^{2p}dV\label{3.54}
\end{equation}
$$
+ \ \frac{c}{K}\int||{\rm Ric}||||\nabla R||
||{\rm Rm}||^{p-2}||\nabla{\rm Rm}||
\eta^{2p}dV
$$
$$
+ \ \frac{c}{K}\int||\nabla R||||{\rm Ric}||||{\rm Rm}||^{p-1}
\eta^{2p-1}||\nabla\eta||dV
$$
$$
\leq \ \ -\frac{1}{6K}
\int||\nabla R||^{2}||{\rm Rm}||^{p-1}
\eta^{2p}dV+\frac{1}{18K}\int||\nabla R||^{2}
||{\rm Rm}||^{p-1}\eta^{2p}dV
+cK B_{2}
$$
$$
+ \ \frac{1}{18K}\int||\nabla R||^{2}
||{\rm Rm}||^{p-1}\eta^{2p}
dV+cK A_{4} \ \ \leq \ \ cK B_{2}+cK A_{4}.
$$
Using (\ref{3.46}) we obtain
$$
\frac{2}{K}\int\langle\langle{\rm Ric},\blacktriangle\widehat{
\bT}\rangle
\rangle||{\rm Rm}||^{p-1}
\eta^{2p}dV \ \ = \ \ \frac{1}{K}
\int\left({\rm Ric}\ast\blacktriangle\widehat{\bT}\right)
||{\rm Rm}||^{p-1}\eta^{2p}dV
$$
$$
= \ \ \frac{1}{K}\int\left(\nabla{\rm Ric}
\ast\nabla\widehat{T}\right)||{\rm Rm}||^{p-1}
\eta^{2p}dV
+\frac{1}{K}\int{\rm Ric}\ast\nabla\widehat{\bT}
\ast\nabla\left(||{\rm Rm}||^{p-1}
\eta^{2p}\right)dV
$$
$$
\leq\frac{c}{K}\int||\nabla{\rm Ric}||
||\nabla\widehat{\bT}||||{\rm Rm}||^{p-1}
\eta^{2p}dV+\frac{c}{K}\int||{\rm Ric}||
||\nabla\widehat{\bT}||||{\rm Rm}||^{p-2}
||\nabla{\rm Rm}||\eta^{2p}dV
$$
$$
+ \ \frac{c}{K}\int||{\rm Ric}||
||\nabla\widehat{\bT}||||{\rm Rm}||^{p-1}
\eta^{2p-1}||\nabla\eta||dV
$$
\begin{equation}
\leq \ \ \frac{1}{50}B_{1}
+c\int||\nabla\bT||^{2}
||{\rm Rm}||^{p-1}
\eta^{2p}dV
+cK B_{2}\label{3.55}
\end{equation}
$$
+ \ c\int||\nabla\bT||^{2}
||{\rm Rm}||^{p-1}\eta^{2p}dV
+cK A_{4}+c\int||\nabla\bT||^{2}||{\rm Rm}||^{p-1}
\eta^{2p}dV
$$
$$
\leq \ \ \frac{1}{50}B_{1}
+cK B_{2}+cK A_{4}
+c\int||\nabla\bT||^{2}||{\rm Rm}||^{p-1}
\eta^{2p}dV
$$
$$
\leq \ \ \frac{9}{50}B_{1}
+cK B_{2}+cK A_{4}+cK^{2}A_{2}
+cK A_{1}
+\frac{d}{dt}\left[\int cR||{\rm Rm}||^{p-1}
\eta^{2p}dV\right].
$$
Similarly, we can prove
$$
\frac{1}{K}\int\left({\rm Ric}\ast\nabla^{2}\widehat{\bT}
\right)||{\rm Rm}||^{p-1}\eta^{2p}dV \ \ = \ \
\frac{1}{K}\int\left(\nabla{\rm Ric}\ast\nabla\widehat{\bT}
\right)||{\rm Rm}||^{p-1}\eta^{2p}dV
$$
$$
+ \ \frac{1}{K}\int{\rm Ric}\ast\nabla\widehat{\bT}
\ast\nabla\left(||{\rm Rm}||^{p-1}
\eta^{2p}\right)dV\ \
\leq \ \ \frac{1}{K}\int\left(\nabla{\rm Ric}
\ast\nabla\widehat{\bT}\right)||{\rm Rm}||^{p-1}
\eta^{2p}dV
$$
$$
+ \ \frac{c}{K}\int||{\rm Ric}||||\nabla\widehat{\bT}||
||{\rm Rm}||^{p-2}||\nabla{\rm Rm}||\eta^{2p}dV
$$
\begin{equation}
+ \ \frac{c}{K}\int||{\rm Ric}||||\nabla\widehat{\bT}||
||{\rm Rm}||^{p-1}\eta^{2p-1}
||\nabla\eta||dV\label{3.56}
\end{equation}
$$
\leq \frac{c}{K}\int||\nabla{\rm Ric}||
||\nabla\widehat{\bT}||||{\rm Rm}||^{p-1}
\eta^{2p}dV+\frac{c}{K}\int||{\rm Ric}||||\nabla
\widehat{\bT}||||{\rm Rm}||^{p-2}||\nabla{\rm Rm}||
\eta^{2p}dV
$$
$$
+ \ \frac{c}{K}\int||{\rm Ric}||||\nabla\widehat{\bT}||
||{\rm Rm}||^{p-1}
\eta^{2p-1}||\nabla\eta||dV
$$
$$
\leq \ \ \frac{9}{50}B_{1}
+cK B_{2}+cK A_{4}+cK^{2}A_{2}
+cK A_{1}
+\frac{d}{dt}\left[\int cR||{\rm Rm}||^{p-1}
\eta^{2p}dV\right].
$$
Plugging (\ref{3.50}) and (\ref{3.53}) -- (\ref{3.56}) into
(\ref{3.49}), and using (\ref{3.46}) and $||\nabla R||^{2}
\leq cK||\nabla\bT||^{2}$, we obtain
$$
B_{1} \ \ \leq \ \ \frac{6}{50}B_{1}+cK B_{2}
+cK A_{4}+cK^{2}A_{2}+cK A_{1}
$$
$$
- \ \frac{d}{dt}\left[\frac{1}{K}\int||{\rm Rm}||^{p-1}
||{\rm Ric}||^{2}\eta^{2p}dV
+c\int(-R)||{\rm Rm}||^{p-1}
\eta^{2p}dV\right]
$$
$$
+ \ \frac{1}{K}\int||\nabla R||^{2}
||{\rm Rm}||^{p-1}\eta^{2p}dV
+\frac{18}{50} B_{1}-\frac{d}{dt}\left[c\int(-R)||{\rm Rm}||^{p-1}
\eta^{2p}dV\right]
$$
$$
\leq \ \ \frac{32}{50}B_{1}+cK B_{2}+cK A_{4}
+cK^{2}A_{2}+cK A_{1}
$$
$$
- \ \frac{d}{dt}\left[\frac{1}{K}\int||{\rm Rm}||^{p-1}
||{\rm Ric}||^{2}\eta^{2p}dV
+c\int(-R)||{\rm Rm}||^{p-1}
\eta^{2p}dV\right].
$$
Thus
\begin{equation}
B_{1} \ \ \leq \ \ cK B_{2}+cK A_{4}+cK^{2}A_{2}+cK A_{1}\label{3.57}
\end{equation}
$$
- \ \frac{d}{dt}\left[\frac{1}{K}\int||{\rm Rm}||^{p-1}
||{\rm Ric}||^{2}\eta^{2p}dV
+c\int(-R)||{\rm Rm}||^{p-1}
\eta^{2p}dV\right]
$$
From (\ref{3.48}) and (\ref{3.57}), we can conclude that

\begin{lemma}\label{l3.5} One has
\begin{equation}
A'_{1} \ \ \leq \ \ cK B_{2}+cK A_{4}+cK^{2}A_{2}+cK A_{1}
\label{3.58}
\end{equation}
$$
- \ \frac{d}{dt}\left[\frac{c}{K}\int||{\rm Rm}||^{p-1}
||{\rm Ric}||^{2}\eta^{2p}dV
+c\int(-R)||{\rm Rm}||^{p-1}
\eta^{2p}dV\right].
$$
\end{lemma}
Observe that two terms in the bracket are both nonnegative, since $R
=-||\bT||^{2}\leq0$.
\\

Finally, we estimate the term $B_{2}$. Using the evolution inequality
\begin{equation*}
||\nabla{\rm Rm}||^{2}
\leq-\frac{1}{2}\blacksquare||{\rm Rm}||^{2}
+c||{\rm Rm}||^{3}+c||\nabla^{2}\bT||||{\rm Rm}||^{3/2}
+c||{\rm Rm}||||\nabla \bT||^{2}
\end{equation*}
we obtain
$$
B_{2} \ \ = \ \ \int||\nabla{\rm Rm}||^{2}
||{\rm Rm}||^{p-3}\eta^{2p}dV \ \ \leq \ \
\int\bigg[-\frac{1}{2}\blacksquare||{\rm Rm}||^{2}
+c||{\rm Rm}||^{3}
$$
$$
+ \ c||\nabla^{2}\bT||||{\rm Rm}||^{3/2}
+c||{\rm Rm}||||\nabla \bT||^{2}\bigg]||{\rm Rm}||^{p-3}
\eta^{2p}dV
$$
\begin{equation}
\leq \ \ -\frac{1}{2}\int\left(\blacksquare||{\rm Rm}||^{2}
\right)||{\rm Rm}||^{p-3}\eta^{2p}dV
+cA_{1}\label{3.59}
\end{equation}
$$
+ \ c\int||\nabla^{2}\bT||||{\rm Rm}||^{p-3/2}\eta^{2p}dV
+c\int||\nabla^{2}\bT||^{2}||{\rm Rm}||^{p-2}
\eta^{2p}dV.
$$
For the first integral one has
$$
-\frac{1}{2}\int\left(\blacksquare||{\rm Rm}||^{2}
\right)||{\rm Rm}||^{p-3}\eta^{2p}dV \ \ = \ \
\frac{1}{2}\int\left(\blacktriangle||{\rm Rm}||^{2}\right)
||{\rm Rm}||^{p-3}\eta^{2p}dV
$$
$$
- \ \frac{1}{2}\int\left(\partial_{t}||{\rm Rm}||^{2}\right)
||{\rm Rm}||^{p-3}
\eta^{2p}dV \ \ = \ \
-\frac{1}{2}\int\left(\partial_{t}||{\rm Rm}||^{2}
\right)||{\rm Rm}||^{p-3}
\eta^{2p}dV
$$
$$
- \ \frac{1}{2}\int\nabla||{\rm Rm}||^{2}
\left[\left(\nabla||{\rm Rm}||^{p-3}
\right)\eta^{2p}+||{\rm Rm}||^{p-3}
\left(\nabla\eta^{2p}\right)\right]dV
$$
$$
= \ \ -\frac{p-3}{4}\int\left(\nabla||{\rm Rm}||^{2}\right)^{2}
||{\rm Rm}||^{p-5}\eta^{2p}dV
$$
$$
+ \ c\int||{\rm Rm}||^{p-2}||\nabla{\rm Rm}||||\nabla
\eta||\eta^{2p-1}dV
-\frac{1}{2}\int\left(\partial_{t}||{\rm Rm}||^{2}
\right)||{\rm Rm}||^{p-3}\eta^{2p}dV
$$
$$
\leq \ \ \frac{1}{50}
B_{2}+cA_{4}-\frac{1}{2}\int\left(\partial_{t}
||{\rm Rm}||^{2}\right)||{\rm Rm}||^{p-3}\eta^{2p}dV.
$$
Here we used the assumption that $p\geq5$. On the other hand,
$$
-\frac{1}{2}\int\left(\partial_{t}||{\rm Rm}||^{2}
\right)||{\rm Rm}||^{p-3}
\eta^{2p}dV \ \ = \ \ -\frac{1}{2}\frac{d}{dt}
\left[\int||{\rm Rm}||^{p-1}\eta^{2p}dV\right]
$$
$$
+ \ \frac{1}{2}\int||{\rm Rm}||^{2}
\left(\partial_{t}||{\rm Rm}||^{p-3}\right)
\eta^{2p}dV
+\frac{1}{2}\int||{\rm Rm}||^{p-1}
\eta^{2p}\left(\partial_{t}dV\right)
$$
$$
\leq \ \ \frac{p-3}{4}\int||{\rm Rm}||^{p-3}
\left(\partial_{t}||{\rm Rm}||^{2}
\right)\eta^{2p}dV+cA_{1}-\frac{1}{2}\frac{d}{dt}
\left[\int||{\rm Rm}||^{p-1}\eta^{2p}dV\right]
$$
so that
$$
-\frac{1}{2}
\int\left(\partial_{t}||{\rm Rm}||^{2}\right)
||{\rm Rm}||^{p-3}\eta^{2p}dV
\leq c A_{1}-\frac{1}{p-1}\frac{d}{dt}
\left[\int||{\rm Rm}||^{p-1}\eta^{2p}dV\right].
$$
Therefore
$$
-\frac{1}{2}\int\left(\blacksquare||{\rm Rm}||^{2}\right)
||{\rm Rm}||^{p-3}\eta^{2p}dV \ \ \leq \ \ \frac{1}{50}B_{2}+c A_{4}+cA_{1}
$$
\begin{equation}
- \ \frac{1}{p-1}\frac{d}{dt}\left[\int||{\rm Rm}||^{p-1}
\eta^{2p}dV\right].\label{3.60}
\end{equation}
To estimate the remainder two integrals, we recall from (\ref{3.9}) that
\begin{equation}
\nabla \bT={\rm Rm}\ast\varphi+\bT\ast \bT\ast\varphi\label{3.61}
\end{equation}
and from (\ref{2.14}) that
\begin{equation}
\nabla\varphi=\bT\ast\psi.\label{3.62}
\end{equation}
From (\ref{3.61}) we get
\begin{equation}
||\nabla \bT||\leq c||{\rm Rm}||+c||\bT||^{2}\leq c||{\rm Rm}||.
\label{3.63}
\end{equation}
In particular, the inequality (\ref{3.63}) yields
\begin{equation}
\int||\nabla \bT||^{2}||{\rm Rm}||^{p-2}
\eta^{2p}dV\leq c\int||{\rm Rm}||^{p}\eta^{2p}dV
\leq cA_{1}.\label{3.64}
\end{equation}
Taking the derivative of (\ref{3.61}) and using (\ref{3.62})
we obtain
\begin{equation}
\nabla^{2}\bT=\nabla{\rm Rm}\ast\varphi
+{\rm Rm}\ast \bT\ast\psi
+\nabla \bT\ast \bT\ast\varphi+\bT\ast \bT\ast \bT\ast\psi.
\label{3.65}
\end{equation}
The particular case $||\nabla^{2}\bT||\leq c||\nabla{\rm Rm}||
+c||{\rm Rm}||||\bT||+c||\nabla \bT||||\bT||+c||\bT||^{3}$ leads to
$$
c\int||\nabla^{2}\bT||||{\rm Rm}||^{p-3/2}\eta^{2p}dV\ \
\leq \ \ c\int\bigg[||\nabla{\rm Rm}||+||{\rm Rm}||||\bT||
+||\nabla \bT||||\bT||
$$
$$
+ \ ||\bT||^{3}\bigg]||{\rm Rm}||^{p-3/2}\phi^{2p}dV \ \
\leq \ \ c\int\left(||\nabla{\rm Rm}||||{\rm Rm}||^{p-3/2}
\eta^{p}\right)
\left(||{\rm Rm}||^{p/2}\eta^{p}
\right)dV
$$
\begin{equation}
+ \ c\int||{\rm Rm}||^{p}\eta^{2p}dV \ \
\leq \ \ \frac{1}{50}B_{2}+cA_{1}.\label{3.66}
\end{equation}
Plugging (\ref{3.60}), (\ref{3.64}), and (\ref{3.66}) into (\ref{3.59}) we arrive at
\begin{equation}
B_{2}\leq cA_{4}+cA_{1}-\frac{d}{dt}\left[\frac{1}{p-1}
\int||{\rm Rm}||^{p-1}
\eta^{2p}dV\right].\label{3.67}
\end{equation}
Together with (\ref{3.58}) and (\ref{3.67}) we finally obtain
$$
(A_{1}+cK A_{2})' \ \ \leq \ \  cK(A_{1}+cK A_{2})
+cK A_{4}
$$
\begin{equation}
- \ \frac{d}{dt}\left[
\frac{c}{K}\int||{\rm Rm}||^{p-1}||{\rm Ric}||^{2}
\eta^{2p}dV
+c\int(-R)||{\rm Rm}||^{p-1}
\eta^{2p}dV\right].\label{3.68}
\end{equation}
Equivalently,

\begin{lemma}\label{l3.6} If $||{\rm Ric}||\leq K$ and $p\geq5$, one has
$$
\frac{d}{dt}
\left[A_{1}+cK A_{2}+\frac{c}{K}\int||{\rm Rm}||^{p-1}
||{\rm Ric}||^{2}\eta^{2p}dV
+c\int(-R)||{\rm Rm}||^{p-1}
\eta^{2p}dV\right]
$$
\begin{equation}
\leq \ \ cK(A_{1}+cK A_{2})+cK A_{4}.\label{3.69}
\end{equation}
\end{lemma}

As in \cite{KMW2016, LY2018}, we choose the domain $\Omega:=
B_{g_{0}}(x_{0},\rho/\sqrt{K})$ and the function
\begin{equation*}
\eta=\left(\frac{\rho/\sqrt{K}-d_{g(0)}(x_{0},\cdot)}{\rho/\sqrt{K}}
\right)_{+}.
\end{equation*}
Then, for all $t\in[0,T]$,
\begin{equation*}
e^{-cK t}g_{0}\leq g(t)\leq e^{cK t}g_{0}, \ \ \
||\nabla_{g(t)}\phi||_{g(t)}\leq e^{cK T}||\nabla_{g_{0}}
\phi||_{g_{0}}\leq\frac{\sqrt{K}e^{cKT}}{\rho}.
\end{equation*}
{\it The proof of Theorem \ref{t1.4}.} Define
\begin{eqnarray}
U&:=&\int||{\rm Rm}||^{p}\phi^{2p}
dV
+cK\int||{\rm Rm}||^{p-1}
\eta^{2p}dV\nonumber\\
&&+ \ \frac{c}{K}\int||{\rm Rm}||^{p-1}
||{\rm Ric}||^{2}\eta^{2p}dV
+c\int(-R)||{\rm Rm}||^{p-1}
\eta^{2p}dV.\label{3.70}
\end{eqnarray}
Then (\ref{3.69}) yields
\begin{equation}
U'\leq cK U+cK A_{4}.\label{3.71}
\end{equation}
For $A_{4}$, using the Young inequality, we have
$$
A_{4} \ \ = \ \ \int||{\rm Rm}||^{p-1}
||\nabla\eta||^{2}\eta^{2p-2}dV \ \leq  \ \int_{B_{g_{0}}
(x_{0},\rho/\sqrt{K})}
||{\rm Rm}||^{p-1}\eta^{2p-2}K\rho^{-2}e^{cKT}dV
$$
$$
\leq \ \ \int_{B_{g_{0}}(x_{0},\rho/\sqrt{K})}
\left[\frac{(||{\rm Rm}||^{p-1}
\eta^{2p-2})^{p/(p-1)}}{\frac{p}{p-1}}
+\frac{(K\rho^{-2}e^{cKT})^{p}}{p}\right]dV
$$
$$
\leq \ \ A_{1}+K^{p}\rho^{-2p}p e^{cKT}
{\rm vol}_{g(t)}\left(B_{g_{0}}
\left(x_{0},\frac{\rho}{\sqrt{K}}\right)\right)
$$
$$
\leq \ \ U+cK^{p}e^{cKT}\rho^{-2p}
{\rm vol}_{g(t)}
\left(B_{g_{0}}\left(x_{0},\frac{\rho}{\sqrt{K}}\right)\right).
$$
Thus
\begin{equation*}
U'\leq cK U+cK^{p+1}e^{cKT}
\rho^{-2-p}{\rm vol}_{g(t)}
\left(B_{g_{0}}\left(x_{0},\frac{\rho}{\sqrt{K}}
\right)\right).
\end{equation*}
As in the proof of \cite{KMW2016}, one can easily deduce from above
that
$$
\int_{B_{g_{0}}(x_{0},\frac{\rho}{2\sqrt{K}})}
||{\rm Rm}_{g(t)}||^{p}_{g(t)}
dV_{g(t)} \ \ \leq \ \ c(1+K) e^{cKT}
\int_{B_{g_{0}}(x_{0},\frac{\rho}{\sqrt{K}})}
||{\rm Rm}_{g_{0}}||^{p}_{g_{0}}
dV_{g_{0}}
$$
\begin{equation}
+ \ c K^{p}\left(1+\rho^{-2p}\right)
e^{cK T}{\rm vol}_{g(t)}\left(B_{g_{0}}
\left(x_{0},\frac{\rho}{\sqrt{K}}\right)\right).\label{3.72}
\end{equation}

As an immediate consequence of the inequality (\ref{3.72}) we give
another proof of the part (a) in Theorem \ref{t1.2}.



\end{document}